\documentclass[11pt,reqno]{amsart}

\renewcommand{\labelenumi}{\theenumi.}

\usepackage[a4paper,left=30mm,right=30mm,top=30mm,bottom=30mm,marginpar=25mm]{geometry}

\usepackage{color}
\usepackage{comment}
\usepackage{amsmath}
\usepackage{amssymb}
\usepackage{amsthm}
\usepackage[latin1]{inputenc}
\usepackage{eurosym}
\usepackage[dvips]{graphics}
\usepackage{graphicx}
\usepackage{epsfig}
\usepackage{dsfont}
\usepackage[displaymath,mathlines]{lineno}
\usepackage{mathtools}
\usepackage[nocompress, space]{cite}

\allowdisplaybreaks

\usepackage{enumitem}
\usepackage{graphicx}
\usepackage{ifthen}

\usepackage{caption}
\usepackage{subcaption}
\usepackage[nocompress, space]{cite}
\usepackage[bookmarksnumbered]{hyperref}
\makeindex

\renewcommand{\div}{\operatorname{div}}

\newcommand{\Rr}{{\mathbb{R}}}

\newcommand{\Nn}{{\mathbb{N}}}

\newcommand{\Tt}{{\mathbb{T}}}

\newcommand{\Li}{L^{\infty}}
\newcommand{\Lip}{{\rm Lip\,}}

\newcommand{\T}{\mathbb{T}}

\newcommand{\epsi}{\varepsilon}

\def\leq{\leqslant}
\def\geq{\geqslant}

\numberwithin{equation}{section}

\newtheoremstyle{thmlemcorr}{10pt}{10pt}{\itshape}{}{\bfseries}{.}{10pt}{{\thmname{#1}\thmnumber{
#2}\thmnote{ (#3)}}}
\newtheoremstyle{thmlemcorr*}{10pt}{10pt}{\itshape}{}{\bfseries}{.}\newline{{\thmname{#1}\thmnumber{
#2}\thmnote{ (#3)}}}
\newtheoremstyle{defi}{10pt}{10pt}{\itshape}{}{\bfseries}{.}{10pt}{{\thmname{#1}\thmnumber{
#2}\thmnote{ (#3)}}}
\newtheoremstyle{remexample}{10pt}{10pt}{}{}{\bfseries}{.}{10pt}{{\thmname{#1}\thmnumber{
#2}\thmnote{ (#3)}}}
\newtheoremstyle{ass}{10pt}{10pt}{}{}{\bfseries}{.}{10pt}{{\thmname{#1}\thmnumber{
A#2}\thmnote{ (#3)}}}

\theoremstyle{thmlemcorr}
\newtheorem{theorem}{Theorem}
\numberwithin{theorem}{section}
\newtheorem{lemma}[theorem]{Lemma}
\newtheorem{corollary}[theorem]{Corollary}
\newtheorem{proposition}[theorem]{Proposition}

\theoremstyle{thmlemcorr*}
\newtheorem{theorem*}{Theorem}
\newtheorem{lemma*}[theorem]{Lemma}
\newtheorem{corollary*}[theorem]{Corollary}
\newtheorem{proposition*}[theorem]{Proposition}
\newtheorem{problem*}[theorem]{Problem}
\newtheorem{conjecture*}[theorem]{Conjecture}

\theoremstyle{defi}
\newtheorem{definition}[theorem]{Definition}

\theoremstyle{remexample}
\newtheorem{remark}{Remark}
\newtheorem{example}{Example}

\newtheorem{pro}[theorem]{Proposition}

\theoremstyle{ass}

\begin{document}

\title[ Discount Mean Field Games]{On weak solutions to first-order discount mean field games}

\author{Hiroyoshi Mitake}
\address[H. Mitake]{
        Graduate School of Mathematical Sciences, The University of Tokyo, 3-8-1 Komaba, Meguro-ku, Tokyo, 153-8914, Japan.}
\email{mitake@ms.u-tokyo.ac.jp}

\author{Kengo Terai}
\address[K. Terai]{
        Graduate School of Mathematical Sciences, The University of Tokyo, 3-8-1 Komaba, Meguro-ku, Tokyo, 153-8914, Japan.}
\email{terai@ms.u-tokyo.ac.jp}

\thanks{
        H. M. was partially supported by the JSPS grants: KAKENHI \#19K03580, \#17KK0093, \#20H01816.  
        K. T. was supported by Grant-in-Aid for JSPS Fellows  \#20J10824. 
}
\date{\today}

\begin{abstract}

In this paper, we establish the existence and uniqueness of weak solutions 
to first-order discount mean field games and a stability result to give the existence for the ergodic problem. We show an example to illustrate the multiplicity of weak solutions to the ergodic problem. With this motivation, we address a selection condition, which is a necessary condition that any limit of solutions under subsequence satisfies. 
As an application, we show a nontrivial example to get the convergence of weak solutions.  
\end{abstract}

\maketitle
\section{introduction}
In this paper, we consider the stationary first-order discount mean field game systems with a local coupling
\begin{equation}\label{DP}
        \begin{cases}
               \:\: \epsi u^\epsi+H(x,Du^\epsi)=f(x,m^\epsi) &\quad\mathrm{in}\  \Tt^d, \\
               \:\: \epsi m^\epsi -\mathrm{div}(m^\epsi D_pH(x, Du^\epsi)) = \epsi &\quad \mathrm{in} \ \Tt^d, \\
        \end{cases}
\end{equation}
where $\Tt^d$ is the $d$-dimensional flat torus identified with $[0,1]^d$, and  $\epsi$ is a given positive number. The functions $H: \Tt^d \times \Rr^d  \to \Rr$
 and $f:\Tt^d \times [0,+\infty) \to \Rr$  are given continuous functions. 
Here, a pair of $u^\epsi: \Tt^d \to \Rr$ and $m^\epsi: \Tt^d \to [0,\infty)$ is unknown.

Mean field game (MFG) systems have been introduced simultaneously by Lasry and Lions \cite{ll1} and by Huang, Caines and Malham\'e \cite{HCM}, which describe the player's optimal strategies of the agents and their macroscopic distribution. 
These games are often determined by a system of a Hamilton-Jacobi equation coupled with a transport or Fokker-Planck equation. In this paper, we focus on deterministic games, and therefore we consider the first order Hamilton-Jacobi equation coupled with a transport equation. 
From a perspective of mathematical analysis, they are not monotone systems, and not uniformly elliptic. 
For first order MFG systems in which the coupling $f$ is of nonlocal nature, the existence and uniqueness of solutions are well-understood (see \cite{ll3}). 
On the other hand, in the case of local couplings, 
in general, we cannot expect the solvability in the classical sense, and therefore it is reasonable to study by introducing the notion of weak solutions. 
Recently the framework of weak solutions has been developed. 
See \cite{ACDPS,Car1,CG,FG2,Gra,P}.  

Our main interest of this paper is to establish the well-posedness of weak solutions to \eqref{DP}, and to study 
the asymptotic behavior of the solution $(u^\epsi,m^\epsi)$ as the discount factor $\epsi \to 0$. 
We call this asymptotic problem the \textit{vanishing discount problem}.
As an analogy of the study of the vanishing discount problem for Hamilton-Jacobi equations 
(see \cite{DFIZ, MT5, IMT}) and also for discount MFG systems (see \cite{GMT}) 
we can naturally expect that the limit problem of \eqref{DP}, which is called the \textit{ergodic problem}, 
is described by 
\begin{equation}\label{EP}
        \begin{cases}
               \:\:  H(x,Du)=f(x,m)+\lambda &\quad\text{in}\  \Tt^d, \\
                \:\: -\mathrm{div}(m D_pH(x, Du))= 0 &\quad \text{in} \ \Tt^d,
        \end{cases}
\end{equation}
where a triple of $u:\Tt^d \to \Rr$, $m:\Tt^d \to [0,\infty)$ and $\lambda\in \Rr$ is unknown. 
We call $\lambda$ an {\it ergodic constant}. 
Ergodic problem \eqref{EP} appears in many context of the asymptotic problem of MFG systems. 
We refer to \cite{Car0,CG} for the study of the long time average of solutions.

The discount problem naturally arises in
optimal control theory and differential
game theory, where $\epsi$  is a discount factor. 
In recent years, there has been much interest and progress on the vanishing discount problem for Hamilton-Jacobi equations. 
Ergodic problem for Hamilton-Jacobi equations is given by 
\begin{equation}\label{EE}
H(x,Dv)=c \quad\mbox{in}\  \Tt^d,
\end{equation}
where a pair of $v: \Tt^d \to \Rr$ and $c \in \Rr$ is unknown. 
One of standard ways to establish the existence of viscosity solutions to \eqref{EE} is to 
consider the solution $v^\epsi \in \mathrm{Lip}(\Tt^d)$ to the discount problem 
\begin{equation*}\label{DD}
\epsi v^\epsi+H(x,Dv^\epsi)=0 \quad\mbox{in}\  \Tt^d, 
\end{equation*}
and pass to the limit. 
Under the coercivity assumption on Hamiltonians, 
we can easily get a priori estimate on $\|Dv^\epsi\|_{\Li(\T^d)}$, and by 
the Arzel\'a-Ascoli theorem, 
we can prove that there exists a subsequence $\{\epsi_j\}_{j\in\mathbb{N}}$ with $\epsi_j\to0$ as $j\to\infty$ such that for some $(v,c)\in\Lip(\T^d)\times\mathbb{R}$, 
\[
\epsi_j v^{\epsi_j}\to -c, \quad
v^{\epsi_j} -\min_{\Tt^d} v^{\epsi_j}\to v \quad\text{in} \ C(\Tt^d) \quad \text{as} \ j\to\infty. 
\]
Here, it is worth emphasizing that due to the lack of uniqueness of viscosity solutions to \eqref{EE}, it is nontrivial whether the whole convergence of $v^\epsi -\min_{\Tt^d} v^\epsi $ holds. 
Recently, 
the convergence to a unique limit and its characterization has been established 
by \cite{DFIZ, MT5, IMT} using weak KAM theory,  the nonlinear adjoint method, and 
the duality method, respectively. 
We also refer \cite{LMT, Hung-book} and the references therein for further development.

For second-order MFG systems, that is, the systems are uniformly elliptic, the vanishing discount problem is studied in \cite{CP,Mas}. 
In first-order MFG systems, in \cite{GMT}, the authors study the existence of classical solutions under the specific Hamiltonian $H(x,p)=\frac{1}{2}|p|^2+V(x)$ with a small oscillatory potential. 
In this setting, since the ergodic problem has the uniqueness up to constants, it is rather easily proved that the uniform convergence holds. 
We also point out that in the argument in \cite{GMT}, the specific form of Hamiltonian is crucial. 
However, if one considers weak solutions, then the multiplicity of weak solutions to \eqref{EP} rather naturally appears. 
In \cite[Section 2.2]{GMT}, the authors consider a weak solution which is introduced in \cite{FG2}, and show an example to illustrate the non-uniqueness issue (see also \cite{GNP}). 
This multiplicity of weak solutions makes the asymptotic problem harder and more interesting.

\medskip
Main feature of this paper, compared particularly to \cite{GMT}, is to study the vanishing discount problem with a different framework of weak solutions which are introduced in \cite{Car1,CG,Gra}. 
In \cite{FG2}, they construct weak solutions based on variational inequality technics and Minty's method, and 
on the other hand, in \cite{Car1,CG,Gra}, they introduce weak solutions by using variational structures and the Fenchel type of duality theorem. For a comparison between these two notions of weak solutions, we
refer the readers to \cite{FG2}.

Main contribution of this paper is firstly to obtain the unique weak solution in the sense of \cite{Car1,CG,Gra} to discount MFG system \eqref{DP}. 
Next, we get a stability result for $(u^\epsi,m^\epsi)$ as $\epsi \to 0$, 
which is a new way to prove the existence of weak solutions to  ergodic problem \eqref{EP}. 
Moreover, we prove that the limit function satisfies a viscosity supersolution property, which will be clearly explained in Section \ref{sta}. 
We also show an example to illustrate the multiplicity of weak solutions to the ergodic problem. With this motivation, we address a selection condition, which shows a necessary condition that any limit of solutions under subsequence satisfies. 
By using this condition, we show a nontrivial example to get the convergence result. 
In connection with the non-uniqueness issue on weak solutions to ergodic problem \eqref{EP},  
we give several uniqueness set results. 
We explain the main results in the next sections in more details.

\subsection {Assumptions}
Throughout the paper we assume the following conditions:
\begin{enumerate}[label=(H\arabic*)]

\item The coupling term $f: \Tt^d \times [0,\infty) \to \Rr$ is continuous in both variables, strictly increasing with respect to the second variable, and there exists $q>1$ and $C>0$ such that 
\begin{equation*}
\frac{1}{C}|m|^{q-1}-C\leq f(x,m)\leq C|m|^{q-1}+C \quad \mbox{ for all } (x,m)\in \Tt^d\times [0,\infty).
\end{equation*}

By replacing $H$, $f$ by $\tilde{H}(x,p):=H(x,p)-\max_{\Tt^d}f(\cdot,0)$, $\tilde{f}(x,m):=f(x,m)-\max_{\Tt^d}f(\cdot,0)$ if necessary, we can always assume $f(x,0)\le 0$ for all $x\in\Tt^d$ without loss of generality. 


\item The Hamiltonian $H: \Tt^d \times \Rr^d \to \Rr$ is continuous in both variables,  
strictly convex and differentiable on the second variable, with $D_pH$ continuous in both variables.
Moreover, there exists $r>1$ and $C>0$ such that
\begin{equation*}
\frac{1}{C}|p|^{r}-C\leq H(x,p)\leq C|p|^{r}+C \quad \mbox{ for all } (x,p)\in \Tt^d\times \Rr^d.
\end{equation*}

\end{enumerate}

\medskip

For later purposes, we define $F: \Tt^d \times \Rr \to \Rr \cup \{+\infty\}$ so that $F(x,\cdot)$ is a primitive of $f(x,\cdot)$ on $(0,+\infty)$, that is,
\begin{align*}
F(x,m):=
\begin{cases}
\int_0^m f(x,s)\mbox{ }ds&\quad \mbox{if }m\geq0\\
+\infty &\quad  \mbox{if }m<0.
\end{cases}
\end{align*}
It follows that $F$ is continuous on $\Tt^d \times (0,\infty)$, differentiable and strictly convex in $m$ and satisfies,  
for some $C>0$, 
\[ 
\frac{1}{C}|m|^{q}-C\leq F(x,m)\leq C|m|^{q}+C \quad \mbox{ for all } (x,m)\in \Tt^d\times [0,\infty).
\]
 Let $F^*$ be the convex conjugate of $F$ with respect to the second variable, i.e.,
\[F^*(x,a)=\sup_{m\in \Rr} \{am-F(x,m)\}=\sup_{m\geq0} \{am-F(x,m)\}.\]
Then, $F^*$ satisfies that, for some $C>0$,  
\[
\frac{1}{C}|a|^{p}-C\leq F^*(x,a)\leq C|a|^{p}+C \quad \mbox{ for all } (x,a)\in \Tt^d\times [0,\infty), 
\]
where $p$ is the conjugate exponent of $q$, that is, $\frac{1}{p}+\frac{1}{q}=1$.

We denote by $H^*$ the convex conjugate of $H$ with respect to the second variable. 
Note that  $H^*$ satisfies 
\begin{equation}\label{Hstar}
\frac{1}{C}|p|^{r'}-C\leq H^*(x,p)\leq C|p|^{r'}+C \quad \mbox{ for all } (x,p)\in \Tt^d\times \Rr^d,
\end{equation}
where $r'$ is the conjugate exponent of $r$. 
Moreover, $H^*(x,\cdot)$ is strictly convex, because $H(x,\cdot)$ is convex, superlinear, and in $C^1(\Rr^d)$ (see \cite[Theorem A.2.4, p. 283]{Cannarsa} for instance) for each $x\in\Tt^d$.

\medskip
We give the definition of weak solutions to \eqref{DP} following to the works in \cite{Car1,Gra,CG}. 

\begin{definition}
  We call a pair of $(u^\epsi,m^\epsi) \in W^{1,pr}(\Tt^d) \times L^q(\Tt^d)$  a weak solution to \eqref{DP} if 
\begin{enumerate}
\renewcommand{\labelenumi}{(\roman{enumi})}

\item[{\rm(i)}] $m^\epsi \geq 0$ a.e. in $\Tt^d$, $\int_{\Tt^d} m^\epsi \mbox{ }dx=1$ and $m^\epsi D_pH(\cdot,Du^\epsi) \in L^1(\Tt^d)$, 

\item[{\rm(ii)}] the first equation of \eqref{DP} holds in the following sense{\rm:} 
\begin{equation}\label{pointwise}
\epsi u^\epsi+H(x,Du^\epsi)=f(x,m^\epsi) \quad \mathrm{a.e.} \mbox{ } \mathrm{in}\  \{m^\epsi>0\}, 
\end{equation}
and
\begin{equation}\label{subsol}
\epsi u^\epsi+H(x,Du^\epsi)\leq f(x,m^\epsi) \quad \mathrm{a.e.} \mbox{ } \mathrm{in}\  \Tt^d,  
\end{equation}

\item[{\rm(iii)}] the second equation of \eqref{DP} holds
\begin{equation}\label{FP} 
\epsi m^\epsi -\mathrm{div}(m^\epsi D_pH(x, Du^\epsi)) = \epsi \quad \mathrm{in} \ \Tt^d 
\end{equation}
in the sense of distribution.
\end{enumerate}

\end{definition}

We notice here that, if $pr>d$,  then $u^\epsi \in C^{0,\gamma}(\Tt^d)$ with $\gamma=1-\frac{d}{pr}$, and $u^\epsi$ is differentiable almost everywhere.

\subsection{Main results}

\setcounter{theorem}{0}

Here, we present main results of  the paper.

\begin{theorem}[Well-posedness]\label{result1}
The discount mean field game system \eqref{DP} has the unique weak solution $(u^\epsi,m^\epsi) \in W^{1,pr}(\Tt^d) \times L^q(\Tt^d)$. Moreover, if $pr>d$, then it holds that 
\begin{equation}\label{supersol}
\epsi u^\epsi+H(x,Du^\epsi)\geq f(x,0) \quad\mathrm{in}\ \Tt^d \quad\mbox{in the sense of viscosity solutions}.
\end{equation}
\end{theorem}

Next, we obtain a weak compactness of $(u^\epsi,m^\epsi)$ and a stability result. 
\begin{theorem}[Stability]\label{result2}
Assume either $q\ge d$, or $r'\leq \frac{qd}{d-q}$ if $q<d$. 
Let $(u^\epsi,m^\epsi) \in W^{1,pr}(\Tt^d) \times L^q(\Tt^d)$ be the weak solution to \eqref{DP}. There exists a subsequence $(u^{\epsi_n},m^{\epsi_n})$ such that 
\begin{align*}
&\langle u^{\epsi_n} \rangle \rightharpoonup u \quad\text{weakly  in} \ W^{1,pr}(\Tt^d), \\
&m^{\epsi_n}\rightharpoonup {m}\quad\text{weakly in} \ L^{q}(\Tt^d), \\
&\epsi_n \int u^{\epsi_n}\,dx \to -\lambda 
\quad \text{as} \ \epsi_n \to 0 
\end{align*}
for some $(u,m, \lambda)\in W^{1,pr}(\Tt^d)\times L^q (\Tt^d)\times \Rr$, which is a weak solution to \eqref{EP} defined by Definition {\rm\ref{def:weak-erg}}. Moreover, if $pr>d$,  we have
 \begin{equation}\label{EPsupersol}
 H(x,Du)\geq f(x,0)+ \lambda \quad\mathrm{in}\ \Tt^d\quad
 \mbox{in the sense of viscosity solutions}.
\end{equation}
Here, we set 
\[
 \langle f\rangle := f(x)-\int_{\Tt^d}f dx
 \]
 for any $f\in L^1(\Tt^d)$. 
\end{theorem}

\medskip
We emphasize here that as in Proposition \ref{par-uni} we have the uniqueness of $(m,\lambda)$, where $(u,m,\lambda)\in  W^{1,pr}(\Tt^d) \times L^q(\Tt^d)\times\mathbb{R}$ is a weak solution to \eqref{EP}. 
However, we have the multiplicity of $u$ in general. 
We show an example to illustrate the multiplicity of weak solutions to \eqref{EP} 
in Section \ref{subsec:lack}. Therefore, it is not clear whether $u^\epsi$ converges or not.  
In the next main theorem, we give a condition which any limit of $\langle u^{\epsi_n} \rangle$ satisfies.

\begin{theorem}[Necessary condition]\label{thm:necessary}
Assume either $q\ge d$, or $r'\leq \frac{qd}{d-q}$ if $q<d$.  
Let  $(u^\epsi,m^\epsi)$ be the weak solution to \eqref{DP}. 
 Assume that $\langle u^\epsi \rangle \rightharpoonup \bar u$ weakly in $W^{1,pr}(\Tt^d)$, $m^\epsi \rightharpoonup m$ weakly in  $L^q(\Tt^d)$ and $\int_{\Tt^d} \epsi u^\epsi\, dx \to- \lambda$ as $\epsi \to 0$. Then,  $\bar u$ is a minimizer of
\begin{equation}\label{criterion}
\inf_{u\in \mathcal{E}}  \int_{\Tt^d}\langle  u \rangle  m \mbox{ }dx,
\end{equation}
where we set 
\begin{equation*}\label{def:e}
\mathcal{E}:=\{
u\in W^{1,pr}(\Tt^d)\mid 
(u,m,\lambda) \ \text{is a weak solution to \eqref{EP}}
\}. 
\end{equation*} 
\end{theorem}
As an application of Theorem \ref{thm:necessary}, we present a nontrivial example that 
we get the whole convergence of $u^\epsi$. 

We also study the uniqueness issue of ergodic problem \eqref{EP}. 
We define $\mathcal{Z} \subset \Tt^d$ by
\begin{equation}\label{uniquenessset}
\mathcal{Z}:=\{x\in \Tt^d\:|\: H(x,0)-\lambda-f(x,0)\geq0 \} \cup \overline{ \{x\in \Tt^d\:|\: m(x)>0 \}}.
\end{equation}
We give a comparison principle on $\mathcal{Z}$ for ergodic problem \eqref{EP}.
\begin{theorem}[Comparison principle]\label{unique1}
Let $(u,m,\lambda)\in (W^{1,pr}(\Tt^d)\cap USC(\Tt^d))\times L^q(\Tt^d) \times \Rr$ be a weak solution to \eqref{EP}. Let $v\in W^{1,pr}(\Tt^d)\cap LSC(\Tt^d)$ satisfy \eqref{EPsupersol}. If  $u\leq v$ on $\mathcal{Z}$, then $u\leq v$ on $\Tt^d$. 
\end{theorem}


Furthermore, we consider the case where 
\begin{equation}\label{cond:con-m}
m\in C(\Tt^d). 
\end{equation}
Under assumption \eqref{cond:con-m}, we have an equivalence between weak solutions and viscosity solutions to \eqref{EP} in the following sense.  

\begin{theorem}[Equivalence]\label{weakvis}
Let $(u,m,\lambda)\in W^{1,pr}(\Tt^d)\times L^q(\Tt^d) \times \Rr$ be a weak solution to \eqref{EP}.
Assume that \eqref{cond:con-m} holds. 
\begin{enumerate}\renewcommand{\labelenumi}{(\roman{enumi})}
\item[{\rm(i)}]
Let $v\in C(\Tt^d)$ be any viscosity solution to
\begin{equation}\label{EPHJ}
H(x,Dv)=f(x,m)+\lambda \quad\text{in}\  \Tt^d.
\end{equation}
 Then, $(v,m,\lambda)$ is a weak solution to \eqref{EP} and satisfies \eqref{EPsupersol}.

\item[{\rm(ii)}] Conversely, if $u$ satisfies \eqref{EPsupersol}, then $u$ is  a viscosity solution to \eqref{EPHJ}.
\end{enumerate}
\end{theorem}
We give sufficient conditions to have \eqref{cond:con-m} in Section \ref{sec:m-conti}.

\bigskip
This paper is organized as follows. In Section \ref{well-pd}, we prove Theorem \ref{result1}, 
that is, we establish the existence and uniqueness of weak solutions to \eqref{DP}. 
In Section \ref{sta}, we investigate weak compactness of $(u^\epsi,m^\epsi)$ and construct a weak solution to the ergodic problem \eqref{EP} as a limit, which implies Theorem \ref{result2}. 
Then, we observe lack of uniqueness of weak solutions for ergodic problem \eqref{EP}  in Section \ref{nonuniqueness}. 
We prove Theorem \ref{thm:necessary} and show a nontrivial example to get a convergence of $u^\epsi$ in Section \ref{selection}. 
In Section \ref{uniquenessstructure}, we prove Theorems \ref{unique1}, \ref{weakvis}.

\section{Well-posedness of discount problem}\label{well-pd}
In this section, we establish the well-posedness result for \eqref{DP} by following the arguments in \cite{Car1,Gra,CG} with a careful modification to the discount problem. 
\subsection{Optimization problems}
We consider two optimization problems corresponding to \eqref{DP}.
Define $\mathcal{A}^\epsi: W^{1,pr}(\Tt^d) \to \Rr$ by
\begin{equation}\label{Aepsi}
\mathcal{A}^\epsi(\phi):=\int_{\Tt^d} F^*(x, \epsi \phi+H(x,D\phi))-\epsi \phi \mbox{ }dx.
\end{equation}
We first notice that $\phi \mapsto \mathcal{A}^\epsi(\phi)$ is convex due to the convexities of $F^{\ast}(x,\cdot)$ and $H(x,\cdot)$. 

Let $K_\epsi$ be the set of pairs $(m,w) \in L^q(\Tt^d)\times   L^{\frac{r'q}{r'+q-1}}(\Tt^d;\Rr^d)$ such that $m\geq 0$ a.e. in $\Tt^d$, and satisfies
\begin{equation*}
\epsi m+\div(w)=\epsi \quad \mbox{in } \Tt^d, \quad \mbox{in the sense of distribution, } 
\end{equation*}
that is,
\begin{equation}\label{DPFP}
\int_{\Tt^d} D\psi\cdot w\mbox{ }dx=\int_{\Tt^d} \epsi(m-1)\psi \mbox{ }dx \quad \text{for all} \  \psi \in C^1(\Tt^d).
\end{equation}

 Note that $r'>1$ and $q>1$ implies $\frac{r'q}{r'+q-1}>1$. We remark that \eqref{DPFP} yields $\int_{\Tt^d} m\mbox{ }dx=1$.

We next define $\mathcal{B}: K_\epsi \to \Rr \cup\{+\infty\}$ by 
\begin{equation}\label{B}
\mathcal{B}(m,w):=\int_{\Tt^d} mH^*(x,-\frac{w}{m})+F(x,m) \mbox{ }dx,
\end{equation}
where if $m(x)=0$, then we set
\begin{align*}
mH^*(x,-\frac{w}{m}):=
\begin{cases}
0&\quad \mbox{if }w(x)=0\\
+\infty &\quad \mbox{if }w(x)\neq0.
\end{cases}
\end{align*}
It is easy to see that 
$(m,w)\mapsto mH^*(x,-\frac{w}{m})$ is convex. 
Indeed, for $(m_1,w_1), (m_2,w_2)\in K_\epsi$ and $t\in (0,1)$, set $(m^t, w^t):= (tm_1+(1-t)m_2,tw_1+(1-t)w_2)$ and then 
\begin{align*}
m^tH^*(x,-\frac{w^t}{m^t})
&=\sup_{p\in \Rr^d} \{ -w^t\cdot p-m^tH(x,p)\}\\
&=\sup_{p\in \Rr^d} \{ t(-w_1\cdot p-m_1H(x,p))+(1-t)(-w_2\cdot p-m_2H(x,p))\}\\
&\leq tm_1H^*(x,-\frac{w_1}{m_1})+(1-t)m_2H^*(x,-\frac{w_2}{m_2}).
\end{align*}
Moreover, noting that $m\mapsto F(x,m)$ and $w\mapsto H^{\ast}(x,-\frac{w}{m})$ are strictly convex 
since $f(x,\cdot)$ is strictly increasing, we see that 
$(m,w)\mapsto \mathcal{B}(m,w)$ is strictly convex. 

\medskip
First, we check that functionals $\mathcal{A}^{\epsi}$ and $\mathcal{B}$ are weakly lower semi-continuous. 
\begin{lemma}\label{lowersemi}
Let $\mathcal{A}^\epsi$ and $\mathcal{B}$ be the functionals defined by \eqref{Aepsi} and \eqref{B}, respectively. 
Then, $\mathcal{A}^\epsi$ and  $\mathcal{B}$ are weakly lower semi-continuous in $W^{1,pr}(\Tt^d)$ and  $L^q(\Tt^d) \times  L^{\frac{r'q}{r'+q-1}}(\Tt^d;\Rr^d)$, respectively. 
\end{lemma}

\begin{proof}
Since $\phi \mapsto \mathcal{A}^\epsi(\phi)$ is convex, it suffices to show that $\mathcal{A}^\epsi$ is lower semi-continuous. Suppose that there exists $\{\phi_n\}_{n\in \Nn} \subset W^{1,pr}(\Tt^d)$ such that $\phi_n \to \phi$ in $W^{1,pr}(\Tt^d)$ and $\liminf_{n \to \infty} \mathcal{A}^\epsi(\phi_n)<\mathcal{A}^\epsi(\phi)$. 
Then, taking a subsequence if necessary, we have $\phi_n(x) \to\phi(x)$ on $\Tt^d$.
 Note that there exists a constant $M>0$ such that
\[F^*(x,\epsi \phi_n(x)+H(x,D\phi_n(x)))-\epsi \phi_n(x)>-M\]
for all $x\in \Tt^d$ and $n \in \Nn$, because $F^*(x,\cdot)$ is nondecreasing and $p>1$. 
It follows from Fatou's lemma that
\begin{align*}
\liminf_{n\to \infty} \mathcal{A}^\epsi(\phi_n)+M&=\liminf_{n\to \infty} \int_{\Tt^d} F^*(x,\epsi \phi_n+H(x,D\phi_n))-\epsi \phi_n +M \mbox{ }dx\\
 &\geq \int_{\Tt^d} F^*(x,\epsi \phi+H(x,D\phi))-\epsi \phi +M \mbox{ }dx=\mathcal{A}^\epsi(\phi)+M,
\end{align*}
which is a contradiction. Hence, \eqref{Aepsi} is lower semi-continuous. Similarly, we can prove that  $\mathcal{B}$ is weakly lower semi-continuous.
\end{proof}
Here, we consider two optimization problems. At first, we study the one corresponding to the first equation of \eqref{DP}. 
\begin{proposition}\label{Aopti}
For all $\epsi>0$, the optimization problem
\begin{equation}\label{Aepsimin}
\inf_{\phi\in W^{1,pr}(\Tt^d)}\mathcal{A}^\epsi(\phi)
\end{equation}
has a minimizer $\phi \in W^{1,pr}(\Tt^d)$. Moreover, if $pr>d$, there exists a minimizer $\bar \phi \in W^{1,pr}(\Tt^d)$ satisfying 
\begin{equation}\label{barphi}
 \epsi \bar \phi+H(x,D\bar \phi)\geq f(x,0) \quad \mbox{ in } \Tt^d \quad
 \mbox{in the sense of viscosity solutions}.
\end{equation}
\end{proposition}

\begin{proof}
Note that $\inf_{\phi \in W^{1,pr}}\mathcal{A}^\epsi(\phi)=\inf_{\phi \in C^{1}}\mathcal{A}^\epsi(\phi)$. 
Take a sequence $\{\phi_n\}_{n\in \Nn} \subset C^{1}(\Tt^d)$ satisfying $\mathcal{A}^\epsi(\phi_n) \to \inf_{\phi \in W^{1,pr}}\mathcal{A}^\epsi(\phi)$ as $n \to \infty$. 
Here, we prove that $\{ \phi_n\}_{n\in \Nn} $ is bounded in $W^{1,pr}(\Tt^d)$. As a preliminary, we check that $\int_{\Tt^d} \phi_n dx$ is bounded.
Note that $F^*(x,a)$ is nondecreasing in $a$. By Jensen's inequality, we have
\begin{align}\label{Aepsibound}
\mathcal{A}^\epsi( \phi_n)
&\geq F^*\left(x, \int_{\Tt^d} \epsi \phi_n+ H(x,D \phi_n)\mbox{ }dx \right)-\int_{\Tt^d} \epsi  \phi_n \mbox{ }dx\notag \\
&\geq F^*\left(x, -C+\int_{\Tt^d} \epsi  \phi_n \mbox{ }dx\right)-\int_{\Tt^d} \epsi  \phi_n \mbox{ }dx\notag \\
&\geq C\left|\int_{\Tt^d} \epsi  \phi_n \mbox{ }dx \right|^p-C-\left|\int \epsi \ \phi_n \mbox{ }dx \right|,
\end{align}
for a sufficient large $C>0$. 
Because $p>1$,  $|\int_{\Tt^d}  \phi_n dx|$ is bounded.
Next, we prove that $D \phi_n$ is bounded in $L^{pr}(\Tt^d)$. We have
\begin{align}\label{Dphiesti}
\| D \phi_n\|_{L^{pr}(\Tt^d)}^{pr} \notag
&=\int_{\Tt^d} ||D \phi_n|^r|^p \mbox{ }dx \leq C \int_{\Tt^d} |H(x,D\phi_n)+C|^p\mbox{ }dx\\ \notag
&\leq C \int_{\Tt^d} |\epsi  \phi_n+ H(x,D \phi_n)|^p+|C-\epsi  \phi_n|^p \mbox{ }dx\\ \notag
&\leq C \int_{\Tt^d} F^*(x,\epsi  \phi_n+H(x,D\phi_n))+C+|C-\epsi  \phi_n|^{p} \mbox{ }dx\\ \notag
&\leq C\left(\mathcal{A}^\epsi( \phi_n)+\int_{\Tt^d} \epsi \phi_n \mbox{ }dx+\|\epsi  \phi_n\|_{L^{p}}^p+1\right)\\ \notag
&\leq C\left(\inf\mathcal{A}^\epsi+1+ \epsi^p \|\langle \phi_n \rangle\|_{L^p(\Tt^d)}^p+C\left|\epsi \int_{\Tt^d}  \phi_n \mbox{ }dx\right|^p+\left|\epsi \int_{\Tt^d}  \phi_n \mbox{ }dx\right|+1\right)\\ \notag
&\leq C \epsi^p \|D \phi_n \|_{L^p(\Tt^d)}^p+C\\ 
&\leq \delta C \epsi^p \|D \phi_n\|^{pr}_{L^{pr}(\Tt^d)}+C_\delta+C,
\end{align}
by using the Poincare-Wirtinger inequality and the Young inequality with arbitrary $\delta>0$ in the last inequality.
By taking a small $\delta>0$, we obtain $\| D \phi_n\|_{L^{pr}(\Tt^d)}\leq C$.

It follows from the Poincare-Wirtinger inequality that $\phi_n$ is bounded in $W^{1,pr}(\Tt^d)$. 
Thus, we can choose a subsequence such that $\phi_n$ converges to some $\phi$ weakly in $W^{1,pr}(\Tt^d)$ as $n\to\infty$. 
In light of Lemma \ref{lowersemi}, $\phi$ is a minimizer of \eqref{Aepsimin}. 

On the other hands, set 
\[\alpha_n(x):=\epsi \phi_n(x)+H(x,D\phi_n(x)).\]
Let $\bar \phi_n$ be the viscosity solution to
\begin{equation*}
\epsi \bar \phi_n+H(x,D\bar \phi_n(x))=\max\{\alpha_n(x),f(x,0)\} \quad\text{in}\  \Tt^d.
\end{equation*}
By comparison, it holds $\bar \phi_n \geq \phi_n$. Noting that 
\begin{equation}\label{Fstar}
F^*(x,a)=0 \quad \text{for all} \  x\in \Tt^d \mbox{ and }  a\leq f(x,0)\le0, 
\end{equation} 
we have 
\begin{align*}
F^*(x,\epsi \bar \phi_n+H(x,D\bar \phi_n(x)))
&=F^*(x,\max\{\alpha_n(x),f(x,0)\})\\
&=\begin{cases}
F^*(x,\alpha_n(x)) &\quad \mbox{if } \alpha_n(x)\geq f(x,0)\\
0 &\quad \mbox{if } \alpha_n(x)\leq f(x,0), 
\end{cases}
\end{align*}
which implies that 
\begin{align*}
\int_{\Tt^d} F^\ast(x,\epsi\bar{\phi}_n+H(x,D\bar\phi_n(x)))\,dx 
&=\,  
\int_{\{\alpha_n\ge f(\cdot,0)\}} F^\ast(x,\epsi\bar{\phi}_n+H(x,D\bar\phi_n(x)))\,dx \\
&=\, 
\int_{\Tt^d} F^\ast(x,\alpha_n(x))\,dx. 
\end{align*}
Thus, we have
\begin{align*}
\mathcal{A}^\epsi(\bar \phi_n)
 &= \int_{\Tt^d} F^*(x,\alpha_n)-\epsi \bar \phi_n \mbox{ }dx\leq \int_{\Tt^d} F^*(x,\alpha_n)-\epsi \phi_n \mbox{ }dx= \mathcal{A}^\epsi(\phi_n).
\end{align*}
Hence, we obtain $\mathcal{A}^\epsi(\bar \phi_n) \to \inf_{\phi \in W^{1,pr}}\mathcal{A}^\epsi(\phi)$ as $n\to \infty$. Similarly, $\bar \phi_n$ is bounded in $W^{1,pr}(\Tt^d)$. Then, we can choose a subsequence such that $\bar \phi_n$ converges to some $\bar \phi$ weakly in $W^{1,pr}(\Tt^d)$. In light of Lemma \ref{lowersemi}, $\bar \phi$ is a minimizer of \eqref{Aepsimin}.
If $pr>d$, by the Rellich-Kondrachov compact embedding theorem, it holds that $\bar \phi_n \to \bar \phi$ uniformly on $\Tt^d$ as $n\to\infty$. Because 
\begin{equation*}
\epsi \bar \phi_n+H(x,D\bar \phi_n(x))\geq f(x,0) \quad\text{in}\  \Tt^d
\end{equation*}
holds in the viscosity sense, due to the stability of viscosity solutions, $\bar \phi$ satisfies \eqref{barphi}.
\end{proof}

Next, we study the dual optimization problems, which corresponds to the second equation of \eqref{DP}.

\begin{proposition}\label{minexist}
The optimization problem
\begin{equation}\label{Bmin}
\inf_{(m,w)\in K_\epsi} \mathcal{B}(m,w)
\end{equation}
has the unique minimizer $(m,w)\in K_\epsi$.
\end{proposition}

\begin{proof}
Take a sequence $(m_n,w_n)\in K_\epsi$ such that $\mathcal{B}(m_n,w_n) \to \inf_{(m,w)\in K_\epsi} \mathcal{B}(m,w)$ as $n\to \infty$. Then, noting that $(1,0)\in K_\epsi$, for sufficiently large $n\in \Nn$, we have
\begin{align}\label{mqesti}
\mathcal{B}(1,0)+1\notag
 \geq \mathcal{B}(m_n,w_n)
 &= \int_{\Tt^d} m_nH^*(x,-\frac{w_n}{m_n})+F(x,m_n)\mbox{ }dx\\ 
& \geq \int_{\Tt^d}\frac{1}{C}|m_n|^q+\frac{m_n}{C}\left|\frac{w_n}{m_n}\right|^{r'}-C \mbox{ }dx.
\end{align}
In particular, $\|m_n\|_{L^q}$ is bounded. 
Note that  $w_n=0$ \textit{a.e.} in $\{m_n=0\}$ because $\mathcal{B}(m_n,w_n)<+\infty$.
By the H\"older inequality, we have
\begin{align}\label{westi}
\int_{\Tt^d} |w_n|^{\frac{r'q}{r'+q-1}}\mbox{ }dx
&=\int_{\{m_n>0\}} |w_n|^{\frac{r'q}{r'+q-1}}\mbox{ }dx \notag \\
&\leq \|m_n\|_{L^q(\Tt^d)}^{\frac{r'-1}{r'+q-1}}\left( \int_{\{m_n>0\}} \frac{|w_n|^{r'}}{m_n^{r'-1}} \mbox{ }dx\right)^{\frac{q}{r'+q-1}}\leq C. 
\end{align}
Hence, we can choose a subsequence such that $(m_n,w_n)\rightharpoonup(m,w)$ weakly in $L^{q}(\Tt^d)\times L^{\frac{r'q}{r'+q-1}}(\Tt^d;\Rr^d)$ as $n\to\infty$. In light of Lemma \ref{lowersemi}, $(m,w)$ is a minimizer of \eqref{Bmin}.

Since $(m,w)\mapsto\mathcal{B}(m,w)$ is strictly convex, $m$ is unique and so is $\frac{w}{m}$ in $\{m>0\}$. As $w=0$ in $\{m=0\}$, uniqueness of $w$ follows as well, 
which implies the uniqueness of the minimizer of \eqref{Bmin}. 
\end{proof}
Finally, we prove that the two optimization problems are in duality.
\begin{proposition}\label{prop:dual}
It holds that
\begin{equation*}
\min_{\phi\in W^{1,pr}(\Tt^d)}\mathcal{A}^\epsi(\phi)=-\min_{(m,w)\in K_\epsi} \mathcal{B}(m,w).
\end{equation*}
\end{proposition}
\begin{proof}
Let  $X:=L^q(\Tt^d) \times L^{\frac{r'q}{r'+q-1}}(\Tt^d;\Rr^d)$. We can rewrite \eqref{Bmin} as 
\begin{equation*}
\min_{(m,w)\in K_\epsi} \mathcal{B}(m,w)
=\min_{(m,w)\in X} \sup_{\phi \in C^1} \int _{\Tt^d} mH^*(x,-\frac{w}{m})+F(x,m)+w\cdot D\phi+\epsi(1-m)\phi \mbox{ }dx.
\end{equation*}
Indeed, because $F(x,m)=+\infty$ for $m<0$, it suffices to pay attention to the case $(m,w)\in X$ does not satisfy \eqref{DPFP}. Then,  there exists $\hat \phi \in C^1(\Tt^d)$ satisfying 
\[ \int_{\Tt^d} w\cdot D\hat \phi+\epsi(1-m)\hat \phi \mbox{ }dx\neq0.\]
Set $\hat \phi_n:= n \hat \phi$, changing the signature if necessary, to yield
\[ \int _{\Tt^d} mH^*(x,-\frac{w}{m})+F(x,m)+w\cdot D\hat \phi_n+\epsi(1-m)\hat \phi_n \mbox{ }dx \to +\infty \quad (n\to \infty),\]
which can not be the infimum. 
Hence, this infimum is attained by $(m,w)\in K_\epsi$ obtained by Proposition \ref{minexist}.

By Sion's min-max theorem, it holds that
\begin{align*}
&\min_{(m,w)\in X} \sup_{\phi \in C^1} \int _{\Tt^d} mH^*(x,-\frac{w}{m})+F(x,m)+w\cdot D\phi+\epsi(1-m)\phi \mbox{ }dx\\
&= \sup_{\phi \in C^1} \min_{(m,w)\in X}  \int _{\Tt^d} mH^*(x,-\frac{w}{m})+F(x,m)+w\cdot D\phi+\epsi(1-m)\phi \mbox{ }dx.
\end{align*}
By using the interchange of minimization and integration (see \cite[Theorem 14.60, p. 677]{R}), 
we obtain
\begin{align*}
&\sup_{\phi \in C^1} \inf_{(m,w)\in X}  \int _{\Tt^d} mH^*(x,-\frac{w}{m})+F(x,m)+w\cdot D\phi+\epsi(1-m)\phi \mbox{ }dx\\
&= \sup_{\phi \in C^1}   \int _{\Tt^d} \inf_{(a,b)\in \Rr \times \Rr^d} aH^*(x,-\frac{b}{a})+F(x,a)+b\cdot D\phi+\epsi(1-a)\phi \mbox{ }dx.
\end{align*}
An easy computation shows that 
\begin{equation*}
F^*(x,\epsi \phi +H(x,D\phi))= -\inf_{(a,b)\in \Rr \times \Rr^d} aH^*(x,-\frac{b}{a})+F(x,a)+b\cdot D\phi
-\epsi a\phi,
\end{equation*}
so that 
\begin{align*}
\min_{(m,w)\in K_\epsi} \mathcal{B}(m,w)
= \sup_{\phi \in C^{1}(\Tt^d)}   \int _{\Tt^d} -F^*(x,\epsi \phi +H(x,D\phi))+\epsi \phi \mbox{ }dx.
\end{align*}
\end{proof}

\subsection{Existence and uniqueness of weak solutions to discount problem}

In this section, we prove that a pair of minimizers of optimization problems \eqref{Aepsimin} and \eqref{Bmin} is a weak solution to  \eqref{DP}. 

\begin{proposition}\label{DPchar}
Let $\phi \in W^{1,pr}(\Tt^d)$ and $(m,w)\in K_\epsi$ be minimizers of \eqref{Aepsimin} and \eqref{Bmin}, respectively. Then, $(\phi,m)$ is a weak solution to \eqref{DP}.

Conversely, if $(\phi,m)$ is a weak solution to \eqref{DP}, then $\phi$  and $(m,w)$ are minimizers of \eqref{Aepsimin} and \eqref{Bmin}, respectively, where we set $w:=-mD_pH(x,D\phi)$. 
\end{proposition}

\begin{proof}
Let $\phi$ and $(m,w)$ be a minimizer of  \eqref{Aepsimin} and \eqref{Bmin}, respectively. 
In view of the duality, Proposition \ref{prop:dual}, we have 
\begin{equation*}
\int_{\Tt^d}F^*(x,\epsi \phi+H(x,D\phi))-\epsi \phi \mbox{ }dx=-\int_{\Tt^d} mH^*(x,-\frac{w}{m})+F(x,m)\mbox{ }dx.
\end{equation*}
By convexities of $F(x,\cdot)$ and $H(x,\cdot)$, we get
\begin{align}
0
&=\int_{\Tt^d} mH^*(x,-\frac{w}{m})+F(x,m)+F^*(x,\epsi \phi+H(x,D\phi))-\epsi \phi \mbox{ }dx \notag\\ 
&\geq \int_{\Tt^d} mH^*(x,-\frac{w}{m})+F(x,m)+m(\epsi \phi+H(x,D\phi))-F(x,m)-\epsi \phi \mbox{ }dx \notag\\ 
&\geq \int_{\Tt^d} -w\cdot D\phi+\epsi \phi (m-1) \mbox{ }dx. \label{kduf}
\end{align}
We now use the fact that $(m,w)$ satisfies \eqref{DPFP},  $D\phi\in L^{pr}(\Tt^d)$ and $w\in L^{\frac{r'q}{r'+q-1}}(\Tt^d)=(L^{pr}(\Tt^d))'$, 
where we denote by $X'$  the dual space of a Banach space $X$. 
They imply
\[\int_{\Tt^d} -w\cdot D\phi+\epsi \phi (m-1) \mbox{ }dx=0.\]
Hence, equalities hold in estimate \eqref{kduf}. In particular, combined with dualities of convexities of $F(x,\cdot)$ and $H(x,\cdot)$, we have
\begin{equation*}\label{FF}
F(x,m)+F^*(x,\epsi \phi+H(x,D\phi))=m(\epsi \phi + H(x,D\phi)) \quad \mathrm{a.e.} \mbox{ } \mathrm{in}\  \Tt^d,
\end{equation*}
and
\begin{equation}\label{TTT}
m\left\{ H^*(x,-\frac{w}{m})+H(x,D\phi)\right\} =-w\cdot D\phi  \quad \mathrm{a.e.} \mbox{ } \mathrm{in}\  \Tt^d.
\end{equation}
Since $F(x,\cdot)$ is strictly convex and smooth on $(0,+\infty)$, we get 
\[\epsi \phi +H(x,D\phi)= \frac{\partial F}{\partial m}=f(x,m) \quad \mbox{ a.e. in } \{m>0\}. \]
Moreover, if $x\in \{m=0\}$, we have
\[ \epsi \phi(x)+H(x,D\phi(x)) \in D^-_mF(x,0),\]
where we write $D^-_mF(x,0)$ for the subdifferential of $F(x,0)$ in $m$. Noting that $D^-_mF(x,0)=(-\infty, f(x,0)]$, we obtain \eqref{subsol}. Also, \eqref{TTT} implies
\[w(x)=-m(x)D_pH(x,D\phi) \quad \mbox{ a.e. in } \Tt^d.\]
By the duality of $H(x,\cdot)$, combined with $(m,w)\in K_\epsi$, we obtain that, for $\psi \in C^1(\Tt^d)$
\[\epsi \int_{\Tt^d} (m-1)\psi \mbox{ }dx=\int_{\Tt^d} w\cdot D\psi \mbox{ }dx=-\int_{\Tt^d} mD_pH(x,D\phi)\cdot D\psi \mbox{ }dx,\]
which yields \eqref{FP}. 

Next, let $(\phi,m)$ be a weak solution to \eqref{DP} and define $w=-mD_pH(x,D\phi)$. Then, $(m,w)$ belongs to $K_\epsi$. Moreover, 
\begin{equation}\label{HHstar}
H^*(x,-\frac{w}{m})=D\phi\cdot D_pH(x,D\phi)-H(x,D\phi).
\end{equation}
We prove that $(m,w)$ is a minimizer for \eqref{Bmin}. Let $(m',w')\in K_\epsi$.
By convexities of $F(x,\cdot)$ and $H^*(x,\cdot)$, using \eqref{HHstar}, we obtain
\begin{align*}
\mathcal{B}(m',w')
&\geq \int_{\Tt^d} m'H^*(x,-\frac{w'}{m'})+F(x,m)+f(x,m)(m'-m)\mbox{ }dx\\ 
&\geq \int_{\Tt^d} m'H^*(x,-\frac{w'}{m'})+F(x,m)+\{\epsi \phi +H(x,D\phi)\}(m'-m)\mbox{ }dx\\ 
&= \int_{\Tt^d} m'\{H^*(x,-\frac{w'}{m'})+H(x,D\phi)\}+F(x,m)-mH(x,D\phi)+\epsi \phi (m'-m)\mbox{ }dx\\ 
&\geq \int_{\Tt^d} -w'\cdot D\phi+F(x,m)-mH(x,D\phi)+\epsi \phi (m'-m)\mbox{ }dx\\ 
&=\int_{\Tt^d} -\epsi \phi(m-1)+F(x,m)-mH(x,D\phi)\mbox{ }dx\\ 
&=\int_{\Tt^d} -\epsi \phi(m-1)+F(x,m)+w\cdot D\phi+mH^*(x,-\frac{w}{m})\mbox{ }dx=\mathcal{B}(m,w). 
\end{align*}
Hence, $(m,w)$ is a minimizer for \eqref{Bmin}.

Finally, we prove that  $\phi$ is a minimizer for \eqref{Aepsimin}. Let $\phi' \in W^{1,pr}(\Tt^d)$.
Note that
\begin{align*}
F^*(x,f(x,m))=\sup_{s\in \Rr}\{ sf(x,m)-F(x,s)\}=mf(x,m)-F(x,m).
\end{align*}
If $m(x)>0$, from \eqref{pointwise} and the above equality, we get
\begin{align}\label{sub:m>0}
F^*(x,\epsi \phi +H(x,D\phi))+F(x,m)=m\{\epsi \phi+H(x,D\phi)\}. 
\end{align}
On the other hand, if $m(x)=0$, from \eqref{subsol}, for all $a\in \Rr$, we have
\begin{align}\label{sub:m=0}
F^*(x,a)-F^*(x,\epsi \phi+H(xD\phi))
&\geq F^*(x,a)-F^*(x,f(x,0))= F^*(x,a)\notag \\
&\geq 0= m(x)\{a-(\epsi \phi+H(x,D\phi))\}.
\end{align}
Thus, \eqref{sub:m>0} and \eqref{sub:m=0} implies $m \in D^-_a F^*(x,\epsi \phi+H(x,D\phi))$. Then, we obtain
\begin{align}\label{u-uni}
\mathcal{A}^\epsi(\phi') \notag
&\geq \int_{\Tt^d} F^*(x,\epsi \phi+H(x,D\phi))+m\{\epsi \phi'-\epsi \phi+H(x,D\phi')-H(x,D\phi)\}-\epsi \phi' \mbox{ }dx\\ \notag
&\geq \int_{\Tt^d} F^*(x,\epsi \phi+H(x,D\phi))+m\left\{\epsi \phi'-\epsi \phi+D_pH(x,D\phi)\cdot D(\phi'-\phi)\right\}-\epsi \phi' \mbox{ }dx\\ \notag 
&= \int_{\Tt^d} F^*(x,\epsi \phi+H(x,D\phi))+m(\epsi \phi'-\epsi \phi)-w\cdot D(\phi'-\phi)-\epsi \phi' \mbox{ }dx\\
&= \int_{\Tt^d} F^*(x,\epsi \phi+H(x,D\phi))-\epsi \phi \mbox{ }dx= \mathcal{A}^\epsi(\phi),
\end{align}
 since $(m,w) \in K_\epsi$. Hence, $\phi$ is a minimizer for \eqref{Aepsimin}.
\end{proof}

Here, we give a proof of Theorem \ref{result1}.

\begin{proof}[Proof of Theorem {\rm\ref{result1}}]
The existence of weak solutions is given by Proposition \ref{DPchar}. Here, we only prove the uniqueness. Let $(u^\epsi_1,m^\epsi_1)$ and $(u^\epsi_2,m^\epsi_2)$ be weak solutions to \eqref{DP}. It follows from the uniqueness of minimizers of \eqref{Bmin} that $m^\epsi:=m^\epsi_1=m^\epsi_2$ a.e. in $\Tt^d$.
Next, suppose that 
\[ \mathcal{L}^d\left(\{x\in \Tt^d\:|\: m^\epsi(x)>0 \mbox{ and } Du_1^\epsi(x) \neq Du_2^\epsi(x)\}\right)>0, \]
where $\mathcal{L}^d(A)$ denotes the Lebesgue measure of $A\subset \Tt^d$. Because $H(x,\cdot)$ is strictly convex, as \eqref{u-uni}, we have
\begin{align*}
\mathcal{A}^\epsi(u^\epsi_2)
&\geq \int_{\Tt^d} F^*(x,\epsi u^\epsi_1+H(x,Du^\epsi_1))+m\{\epsi u^\epsi_2-\epsi u^\epsi_1+H(x,Du^\epsi_2)-H(x,Du^\epsi_1)\}-\epsi u^\epsi_2 \mbox{ }dx\\
&> \int_{\Tt^d} F^*(x,\epsi u^\epsi_1+H(x,Du^\epsi_1))+m\{\epsi u^\epsi_2-\epsi u^\epsi_1+D_pH(x,Du^\epsi_1)\cdot D(u^\epsi_2-u^\epsi_1)\}-\epsi u^\epsi_2 \mbox{ }dx\\
&\geq \int_{\Tt^d} F^*(x,\epsi u^\epsi_1+H(x,Du^\epsi_1))-\epsi u^\epsi_1 \mbox{ }dx= \mathcal{A}^\epsi(u^\epsi_1).
\end{align*}
Hence, $u^\epsi_2$ is not a minimizer of \eqref{Aepsimin}, which contradicts the result of Proposition \ref{DPchar}. Thus, we conclude that $Du^\epsi_1=Du^\epsi_2$ a.e. in  $\{x\in \Tt^d \:| \: m^\epsi(x)>0\}$. 
In particular, by  \eqref{pointwise}, $u^\epsi_1=u^\epsi_2$ a.e. in $\{x\in \Tt^d \:| \: m^\epsi(x)>0\}$.

Next, we prove the uniqueness of $u^\epsi$. We define
\begin{equation*}
\bar u^\epsi(x) := \max\{ u^\epsi_1(x),u^\epsi_2(x)\}=\frac{1}{2} (u^\epsi_1+u^\epsi_2+|u^\epsi_1-u^\epsi_2|).
\end{equation*}
As in \cite[Lemma 7.6]{GilTru}, it follows $\bar u^\epsi \in W^{1,pr}(\Tt^d)$ with
\begin{align*}
D\bar u^\epsi = \chi_{\{ u^\epsi_1>u^\epsi_2\}} Du^\epsi_1+\chi_{\{ u^\epsi_1<u^\epsi_2\}}Du_2^\epsi+\frac{1}{2}\chi_{\{ u^\epsi_1=u^\epsi_2\}}(Du^\epsi_1+Du^\epsi_2) \quad \mathrm{a.e.} \mbox{ } \text{in}\  \Tt^d,
\end{align*}where $\chi$ is the characteristic function.
We claim that $(\bar u^\epsi,m^\epsi)$ is a  weak solution to \eqref{DP}. Indeed, because $D\bar u^\epsi=Du^\epsi_1=Du^\epsi_2$ a.e. in $\{x\in \Tt^d \:| \: m^\epsi(x)>0\}$, it is clear that \eqref{pointwise} and \eqref{FP} hold. Moreover, we can see that 
\[\epsi \bar u^\epsi+H(x,D \bar u^\epsi)\leq f(x,m^\epsi) \quad \mathrm{a.e.} \mbox{ } \text{in}\  \{u^\epsi_1 \neq u^\epsi_2\}. \]
For almost all $x  \in \{ u^\epsi_1 = u^\epsi_2\}$, by the convexity of the Hamiltonian, we get
\begin{align*}\label{ineq:ubar}
\epsi \bar u^\epsi+H(x,D \bar u^\epsi)
&\leq \frac{1}{2} (\epsi u^\epsi_1+H(x, Du^\epsi_1))+\frac{1}{2} (\epsi u^\epsi_2+H(x, Du^\epsi_2))\leq f(x,m^\epsi),
\end{align*}
which implies $(\bar u^\epsi,m^\epsi)$ is a weak solution to \eqref{DP}.

Note that $x\in\{m^\epsi=0\}$, by \eqref{subsol} and \eqref{Fstar}, we have
\begin{equation}\label{eq:m=0}
 \int_{\{m^\epsi=0\}}F^*(x,\epsi u^\epsi_1+H(x,Du^\epsi_1)) \mbox{ }dx=\int_{\{m^\epsi=0\}} F^*(x,\epsi \bar u^\epsi+H(x,D\bar u^\epsi))\mbox{ }dx .
\end{equation}
Here, we suppose that $\mathcal{L}^d(\{u^\epsi_1 \neq u^\epsi_2\})>0$. By \eqref{pointwise} and \eqref{eq:m=0}, we have
\begin{align*}
\mathcal{A}^\epsi(\bar u^\epsi)
& < \int_{\Tt^d} F^*(x, \epsi \bar u^\epsi+H(x,D\bar u^\epsi))-\epsi u^\epsi_1 \mbox{ }dx\\
&=\int_{\Tt^d} F^*(x,\epsi u^\epsi_1+H(x,Du^\epsi_1))-\epsi u^\epsi_1 \mbox{ }dx = \mathcal{A}^\epsi(u^\epsi_1),
\end{align*}
which contradicts that $u^\epsi_1$ is a minimizer of \eqref{Aepsimin}. Thus, we have $u^\epsi_1=u^\epsi_2$ on $\Tt^d$.

Finally, we assume $pr>d$. Because of the uniqueness of weak solutions, $u^\epsi$ coincides the minimizer $\bar \phi$ obtained in Proposition \ref{Aopti}, which satisfies \eqref{supersol} in the viscosity sense. 
\end{proof}

\section{weak compactness and stability}\label{sta}
In this section we prove Theorem \ref{result2}. We first recall the definition of weak solutions to \eqref{EP} introduced by \cite{CG}.

\begin{definition}\label{def:weak-erg}
  We call a triple of $(u,m,\lambda) \in W^{1,pr}(\Tt^d) \times L^q(\Tt^d)\times \Rr$  a weak solution to \eqref{EP} if 
\begin{enumerate}

\item[{\rm(i)}] $m\geq 0$ a.e. in $\Tt^d$, $\int_{\Tt^d} m\mbox{ }dx=1$ and $mD_pH(\cdot,Du) \in L^1(\Tt^d)$,

\item[{\rm(ii)}] the first equation of \eqref{EP} holds in the following sense:
\begin{equation}\label{EPpointwise}
H(x,Du)=f(x,m)+\lambda \quad \mathrm{a.e.} \mbox{ } \mathrm{in}\  \{m>0\}, 
\end{equation}
and
\begin{equation}\label{EPsubsol}
H(x,Du)\leq f(x,m)+\lambda  \quad \mathrm{a.e.} \mbox{ } \mathrm{in}\  \Tt^d,  
\end{equation}

\item[{\rm(iii)}] the second equation of \eqref{EP} holds
\begin{equation}\label{EPFP}
-\mathrm{div}(m D_pH(x, Du)) = 0 \quad \mathrm{in} \ \Tt^d, 
\end{equation}
in the sense of distribution.

\end{enumerate}
\end{definition}

\subsection{Discounted aproximations}
The existence of weak solutions to \eqref{EP} is firstly proved by \cite{CG}. In \cite{CG}, they directly prove the existence by considering the optimization problem as in Section \ref{well-pd}. In this paper, we prove it by considering weak compactness and stability for the vanishing discount problem. 

\begin{proposition}\label{compact}
Let $(u^\epsi,m^\epsi)\in W^{1,pr}(\Tt^d)\times L^q(\Tt^d)$ be the weak solution to \eqref{DP}, and  $w^\epsi=-m^\epsi D_pH(x,Du^\epsi)$. There exists a constant $C>0$ independent of $\epsi>0$ such that
\begin{equation}\label{wcom}
\|\langle u^\epsi \rangle\|_{W^{1,pr}(\Tt^d)}+\|m^\epsi\|_{L^q(\Tt^d)}+\|w^\epsi\|_{L^{\frac{r'q}{r'+q-1}}(\Tt^d;\Rr^d)}+\left| \epsi \int_{\Tt^d}u^\epsi \mbox{ }dx \right| \leq C.
\end{equation}
\end{proposition}

\begin{proof}
At first, we can see that $|\epsi \int u^\epsi dx|$ is bounded as \eqref{Aepsibound}.
Next, as in \eqref{Dphiesti}, we get the bound for $\|Du^\epsi\|_{L^{pr}}$. It follows from  the Poincare-Wirtinger inequality that $\|\langle u^\epsi \rangle\|_{W^{1,pr}}$ is bounded uniformly in $\epsi$.

Finally, using \eqref{mqesti} and \eqref{westi}, we have 
\begin{equation*}
\|m^\epsi\|_{L^q(\Tt^d)}+\|w^\epsi\|_{L^{\frac{r'q}{r'+q-1}}(\Tt^d;\Rr^d)} \leq C,
\end{equation*}which gives \eqref{wcom}.
\end{proof}

Finally, we study the stability of weak solutions using  $\Gamma$-convergence type arguments. To prove Theorem \ref{result2}, we introduce several notations. We define  $\mathcal{A}: W^{1,pr}(\Tt^d)\times \Rr \to \Rr$ by
\begin{equation*}
\mathcal{A}(\phi,\lambda):=\int_{\Tt^d} F^*(x, -\lambda+H(x,D\phi))+\lambda \mbox{ }dx.
\end{equation*}

Let $K$ be the set of pairs $(m,w) \in L^q(\Tt^d)\times   L^{\frac{r'q}{r'+q-1}}(\Tt^d;\Rr^d)$ such that $m\geq 0$ a.e. in $\Tt^d$, $\int _{\Tt^d}m=1$, and satisfies
\begin{equation*}
\div(w)=0,
\end{equation*}
in the sense of distribution. Next, we prove the following two lemma.

\begin{lemma}\label{Asup}
It holds that
\begin{equation*}\label{Asup}
\limsup_{\epsi \to 0} \min_{\phi \in W^{1,pr}(\Tt^d)}\mathcal{A}^\epsi(\phi)\leq  \inf_{(\phi,c) \in W^{1,pr}(\Tt^d)\times \Rr}\mathcal{A}(\phi,c).
\end{equation*}
\end{lemma}

\begin{proof}
Take any $(\phi,c)\in C^1(\Tt^d)\times \Rr$. Set $\phi^\epsi:=\phi-c/\epsi$. Then, we get
\begin{align*}
\limsup_{\epsi \to 0} \inf_{\phi \in W^{1,pr}(\Tt^d)}\mathcal{A}^\epsi(\phi)
&\leq \limsup_{\epsi \to 0}  \int_{\Tt^d} F^*(x,\epsi \phi^\epsi+H(x,D\phi^\epsi))-\epsi \phi^\epsi \mbox{ }dx\\
&= \limsup_{\epsi \to 0}  \int_{\Tt^d} F^*(x,\epsi \phi-c+H(x,D\phi))-\epsi \phi+c \mbox{ }dx\\
&\leq \int_{\Tt^d} F^*(x,-c+H(x,D\phi))+c \mbox{ }dx=\mathcal{A}(\phi,c).
\end{align*}
Taking the infimum on $(\phi,c)$ yields the conclusion.
\end{proof}

\begin{lemma}\label{limsupB}
Assume either $q\ge d$ or $r'\leq \frac{qd}{d-q}$ if $q<d$. Then, it holds that
\begin{equation}\label{Bsup}
\limsup_{\epsi \to 0} \min_{(m,w) \in K_\epsi}\mathcal{B}(m,w)\leq  \inf_{(m,w) \in K}\mathcal{B}(m,w).
\end{equation}
\end{lemma}

\begin{proof}
We only need to consider the case where $\inf_{K}\mathcal{B}<+\infty$. 
Take any $(m_n, w_n)\in K$ so that $\mathcal{B}( m_n, w_n)\to \inf_{K}\mathcal{B}$ as $n\to\infty$. 
Consider 
\begin{equation*}\label{poisson}
-\Delta v_n= m_n-1 \quad \mbox{in } \Tt^d.
\end{equation*}
Since $m_n\in L^q(\Tt^d)$, we have $v_n\in W^{2,q}(\Tt^d)$. In particular, $v_n$ satisfies
\[ 
\int_{\Tt^d} D\psi \cdot D v_n \mbox{ }dx=\int_{\Tt^d} (m_n-1) \psi \mbox{ }dx, \quad \mbox{for all } \psi \in C^1(\Tt^d).  
\]
Set
\[
m^\epsi_n:= \epsi+(1-\epsi)m_n, \quad \mbox{and}\quad  w^\epsi_n:=w_n-\epsi(\epsi-1)Dv_n.
\]
Note that if $q\ge d$ then we can easily check that 
$(m^\epsi_n,w^\epsi_n)\in L^q(\Tt^d)\times L^{\frac{r'q}{r'+q-1}}(\Tt^d)$. 
If $q<d$ then since we assume $r'\leq \frac{qd}{d-q}$, we have 
\[
\frac{r'q}{r'+q-1}\le \frac{q^2d}{(d-q)(r'+q-1)}<\frac{qd}{d-q}.   
\]
By the Sobolev inequality we get  
$w^\epsi_n\in L^{\frac{r'q}{r'+q-1}}(\Tt^d)$. 
Moreover, 
it is easy to check that 
\begin{align*}
\int_{\Tt^d} w^\epsi_n \cdot D\psi \mbox{ }dx
&=\int_{\Tt^d}  w_n\cdot D\psi-\epsi(\epsi-1)D\psi \cdot Dv_n \mbox{ }dx\\
&=\int_{\Tt^d} -\epsi(\epsi-1)( m_n-1)\psi \mbox{ }dx=\int_{\Tt^d} \epsi(m^\epsi_n-1)\psi\mbox{ }dx 
\end{align*}
for all $\psi \in C^1(\Tt^d)$, 
which implies $(m^\epsi_n, w^\epsi_n) \in K_\epsi$.


Note that $F(x,m^\epsi_n)\le C(|m_n|^q+1)\in L^1(\Tt^d)$, and 
\begin{align}\label{ineq:L1}
m^\epsi_n H^*\left(x,-\frac{ w^\epsi_n}{ m^\epsi_n}\right) 
\leq C \left(\frac{|w_n-\epsi(\epsi-1)Dv_n|^{r'}}{|\epsi+(1-\epsi)m_n|^{r'-1}}+ m^\epsi_n\right)
\leq C \left( m_n \left|\frac{w_n}{m_n}\right|^{r'}+|Dv_n|^{r'}+  m_n \right).  
\end{align}
By a similar argument to Proposition \ref{minexist} and using assumption $r'\leq \frac{qd}{d-q}$, we can see that the right hand side of \eqref{ineq:L1} is in $L^1(\Tt^d)$.

By the Fatou lemma, we obtain 
\begin{align*}
\limsup_{\epsi \to 0} \min_{(m,w) \in K_\epsi}\mathcal{B}(m,w)
&\leq  \limsup_{\epsi \to 0} \int_{\Tt^d}  m^\epsi_n H^*(x,-\frac{ w^\epsi_n}{ m^\epsi_n})+F(x, m^\epsi_n)\mbox{ }dx\\
&\leq \int_{\Tt^d}  m_n H^*(x,-\frac{ w_n}{m_n})+F(x,m_n)\mbox{ }dx =\mathcal{B}(m_n, w_n).
\end{align*}
Sending $n\to\infty$ yields the conclusion. 
\end{proof}

We are ready to prove Theorem \ref{result2}.

\begin{proof}[Proof of Theorem {\rm\ref{result2}}]
By Lemmas \ref{Asup}, \ref{limsupB} and the duality on $\mathcal{A}$ and $\mathcal{B}$ by \cite[Lemma 4.3]{CG}, we have
\begin{align*}
&\limsup_{\epsi \to 0} \inf_{\phi \in W^{1,pr}(\Tt^d)}\mathcal{A}^\epsi(\phi)
\leq  \inf_{(\phi,c) \in W^{1,pr}(\Tt^d)\times \Rr}\mathcal{A}(\phi,c)=-\inf_{(m,w)\in K} \mathcal{B}(m,w)\\
&\leq -\limsup_{\epsi \to 0} \inf_{(m,w) \in K_\epsi} \mathcal{B}(m,w)
=\liminf_{\epsi \to 0} \left(-\inf_{(m,w) \in K_\epsi} \mathcal{B}(m,w) \right)
=\liminf_{\epsi \to 0} \inf_{\phi \in W^{1,pr}(\Tt^d)}\mathcal{A}^\epsi(\phi).
\end{align*}
Thus, equalities hold in the above estimates, which implies
\begin{align*}
\lim_{\epsi \to 0} \inf_{\phi \in W^{1,pr}(\Tt^d)}\mathcal{A}^\epsi(\phi)&= \inf_{(\phi,c) \in W^{1,pr}(\Tt^d)\times \Rr}\mathcal{A}(\phi,c)\\
&=-\inf_{(m,w)\in K} \mathcal{B}(m,w)=-\lim_{\epsi \to 0} \inf_{(m,w) \in K_\epsi}\mathcal{B}(m,w).
\end{align*}

Let $(m^\epsi,w^\epsi)$ be the minimizer of \eqref{Bmin}. By Proposition \ref{compact}, with choosing subesequences, $(m^\epsi,w^\epsi)$ converges weakly in $L^q(\Tt^d) \times   L^{\frac{r'q}{r'+q-1}}(\Tt^d;\Rr^d) $ for some $(m,w)\in K$. Then, we observe
\[\inf_{(m,w) \in K} \mathcal{B}(m,w)= \lim_{\epsi \to 0} \inf_{(m,w)\in K_\epsi} \mathcal{B}(m,w)=\liminf_{\epsi \to 0} \mathcal{B}(m^\epsi,w^\epsi)\geq \mathcal{B}(m,w),\]
which implies that $(m,w)$ minimizes $\mathcal{B}$ among $K$.

On the other hand, let $u^\epsi$ be the minimizer of \eqref{Aepsi}. In view of Proposition \ref{compact},  $\langle u^\epsi \rangle \rightharpoonup u$ weakly in $W^{1,pr}(\Tt^d)$ and $\int \epsi u^\epsi \to -\lambda$ as $\varepsilon\to0$ along subsequences. Similarly, we can see that $(u,\lambda)$ minimize $\mathcal{A}$ among $W^{1,pr}(\Tt^d) \times \Rr$. By \cite[Proposition 4.5]{CG}, $(u,m,\lambda)$ is a weak solution to \eqref{EP}.

Finally, if we assume $pr>d$,
by the Rellich-Kondrachov compact embedding theorem, $\langle u^\epsi \rangle \to u$ uniformly on $\Tt^d$. In view of the stability of viscosity solutions and \eqref{supersol}, $u$ satisfies \eqref{EPsupersol} in the viscosity sense. 
\end{proof}

\subsection{Continuity of density $m$}\label{sec:m-conti}
In this section, we give some conditions to obtain the continuity of $m$. 
\begin{proposition}
Let $d=1$. Assume that there exists $\tilde{J}:[0,\infty) \to \Rr$ and $c_0>0$ such that 
\begin{equation}\label{exF}
 F(x,m)-F^*(x,a)\geq ma+c_0|m-\tilde{J}(a)|^2 \quad \text{for all} \  m,a\in [0,\infty) \mbox{ and } x\in \Tt.
\end{equation}
Then, we have 
\begin{equation}\label{esti:sobolev}
m\in H^1(\Tt)\subset C(\Tt). 
\end{equation}
 Furthermore, assume that $H^*$ and $F$ are twice differentiable in the first variable and 
\begin{equation*}\label{HxxFxx}
H^*_{xx}(x,p) \leq C (|p|^{r'}+1) \quad \mbox{ and }\quad F_{xx}(x,m) \leq C (|m|^q+1),
\end{equation*}
for all $(x,p)\in \Tt \times \Rr$ and $m \geq0$. 
Then, we have 
\begin{equation}\label{esti:m-ep}
\|m^\epsi\|_{H^1(\Tt)}\le C 
\end{equation}
for some $C\ge0$, which is independent of $\epsi$. 
\end{proposition}

\begin{proof}
The continuity of $m$, that is \eqref{esti:sobolev}, is due to \cite[Theorem 3.4]{GM}. 
Estimate \eqref{esti:m-ep} is a straightforward result of a Sobolev estimate in Appendix, 
Proposition \ref{H1esti}. 
\end{proof}

For instance, if we have $F(x,m)=F(m)$ and $F''(m)\geq c_0>0$ for all $(x,m)\in\Tt\times[0,\infty)$, 
then taking $\tilde{J}(a):=(F^*)'(a)=f^{-1}(a)$, we can check that \eqref{exF} holds. 
See \cite[Section 3]{Santambrogio} for more details. 

\begin{remark}
We point out here that due to \eqref{esti:m-ep} and the uniqueness of $m$ we have 
\begin{equation*}
m^\epsi\to m\quad\text{in} \ C(\Tt) \quad\text{as} \ \epsi\to0, 
\end{equation*}
which is independently of interest. 
\end{remark}

In a higher dimension, we give another type of results.  
\begin{proposition}\label{m:const}
Let $c\in \Rr$ and assume that $(c,m,\lambda)$ is a weak solution to ergodic problem \eqref{EP}. Then, $m\in C(\Tt^d)$.
\end{proposition}

\begin{proof}
From \eqref{EPpointwise}, it holds that 
\[H(x,0)=f(x,m)+\lambda \quad  \mbox{ a.e. in } \{m>0\}. \]
Because  $\lambda \mapsto \int_{\Tt^d} \max\{f^{-1}(x,H(x,0)-\lambda),0\} \mbox{ }dx$ is monotone, we can  choose $\lambda$ so that 
\[ \int_{\Tt^d} \max\{f^{-1}(x,H(x,0)-\lambda),0\} \mbox{ }dx=1.\] 
Hence, we can explicitly get
\[m(x)=\max\{f^{-1}(x,H(x,0)-\lambda),0\}.\]
\end{proof}

For instance, letting $H(x,p):=\frac{1}{2}|p|^2+V(x)$ for some $V\in C(\Tt^d)$, we can easily check that $(0,m,\lambda)$ is a weak solution to ergodic problem \eqref{EP}.

%

\section{Uniqueness issue of the ergodic problem}\label{nonuniqueness}
Here, we study uniqueness and non-uniqueness of solutions to the ergodic problem \eqref{EP}. We prove that $(m,\lambda)\in L^q(\Tt^d)\times \Rr$, which is a part of weak solutions to \eqref{EP}, is unique. On the other hand, we give an example which shows the multiplicity of $u\in W^{1,pr}(\Tt^d)$.

\subsection{Uniqueness of $m$ and $\lambda$}
First, we prove the uniqueness of $(m,\lambda)$. This is known in \cite{CG}, but we give it for the completeness of the paper. 
\begin{proposition}\label{par-uni}
Let $(u_1,m_1,\lambda_1)$ and $(u_2,m_2,\lambda_2)$ be weak solutions to \eqref{EP}. Then,  $\lambda_1=\lambda_2$, $m_1=m_2$ a.e. in $\Tt^d$ and $Du_1=Du_2$ a.e. in $\{x\in \Tt^d \:|\: m(x)>0\}$.
\end{proposition}

\begin{proof}

 As in Proposition \ref{minexist}, $\inf_{(m,w)\in K}\mathcal{B}(m,w)$ has the unique minimizer $(m,w)\in K$. As in the proof of Proposition \ref{DPchar}, we have 
\[w=-mD_pH(x,Du_i), \]
and combining with \eqref{EPpointwise},
\begin{align*}
-w\cdot Du_i
&=m\{H^*(x,-\frac{w}{m})+H(x,Du_i)\}=m\{H^*(x,-\frac{w}{m})+(f(x,m)+\lambda_i)\},
\end{align*}
for $i=1,2$. Then, we have
\[m\lambda_i= -mH^*(x,-\frac{w}{m})-w\cdot Du_i-mf(x,m) \quad \mbox{a.e.} \mbox{ } \mathrm{in}\  \Tt^d.\]
Therefore, integrating over $\Tt^d$, we get 
\begin{align*}
\lambda_i
&=\int_{\Tt^d} -mH^*(x,-\frac{w}{m})-w\cdot Du_i-mf(x,m)\mbox{ }dx\\
&=\int_{\Tt^d} -mH^*(x,-\frac{w}{m})-mf(x,m)\mbox{ }dx,
\end{align*}
since we have $\div(w)=0$ in the sense of distribution, which implies $\lambda_1=\lambda_2$. Set $\lambda:=\lambda_1=\lambda_2$.

Next,
 Suppose that 
\[ \mathcal{L}^d\left(\{x\in \Tt^d\:|\: m(x)>0 \mbox{ and } Du_1(x) \neq Du_2(x)\}\right)>0. \]
In light of the strict convexity of $H(x,\cdot)$, by a similar argument to that of the proof of Proposition \ref{DPchar}, we have
\begin{align*}
\mathcal{A}(u_2,\lambda)
&\geq \int_{\Tt^d} F^*(x,-\lambda+H(x,Du_1))+m(H(x,Du_2)-H(x,Du_1))+\lambda \mbox{ }dx.\\
&> \int_{\Tt^d} F^*(x,-\lambda+H(x,Du_1))+mD_pH(x,Du_1)\cdot D(u_2-u_1)+\lambda \mbox{ }dx\\
&= \int_{\Tt^d} F^*(x,-\lambda +H(x,Du_1))+\lambda \mbox{ }dx= \mathcal{A}(u_1,\lambda),
\end{align*}
where we use the fact that $(m,w)$ satisfies \eqref{EPFP}.
Hence, $(u_2,\lambda)$ does not minimize $\mathcal{A}$ among $W^{1,pr}(\Tt^d)\times \Rr$, which contradicts the result in \cite[Proposition 4.5]{CG}. Thus we conclude that $Du_1=Du_2$ a.e. in $\{x\in \Tt^d \:| \: m(x)>0\}$.
\end{proof}

\subsection{Lack of uniqueness}\label{subsec:lack}
In this section, we point out an example to demonstrate the non-uniqueness issue of $u\in W^{1,pr}(\Tt^d)$, where $(u,m,\lambda)$ is a weak solution to \eqref{EP} satisfying \eqref{EPsupersol}.

\begin{example}\label{ex:1}
Let $d=1$, $f(x,m)=m$ and $H(x,p)=\frac{1}{2}|p|^2+W(x)$, 
where $W:\Tt \to \Rr$ is given by
\begin{align*}\label{ex:W}
W(x):=
\begin{cases}
-32x+4 &\quad \mbox{for  } \:\:0\leq x\leq \frac 1 4\\
32x-12 &\quad \mbox{for  } \:\: \frac 1 4\leq x\leq \frac 1 2\\
-32x+20 &\quad \mbox{for  } \:\: \frac1 2\leq x\leq \frac 3 4\\
32x-28   &\quad \mbox{for  } \:\: \frac 3 4\leq x\leq 1.
\end{cases}
\end{align*}

Then, \eqref{EP} becomes
\begin{equation}\label{example4}
        \begin{cases}
                \:\:\frac{1}{2}|u_x|^2+W(x)=m+\lambda &\quad\text{in}\  \Tt, \\
                \:\: -(m u_x)_x= 0 &\quad \text{in} \ \Tt.
        \end{cases}
\end{equation} In this case, $p=q=r=2$. Thus $u\in W^{1,4}(\Tt)\subset C^{0,\frac{3}{4}}(\Tt)$.
Set $\lambda=0$ and $m(x)=\max\{W(x),0\}$. Then, $(0,m,\lambda)$ is a weak solution to \eqref{example4}. 
Indeed, clearly \eqref{EPpointwise} and \eqref{EPsubsol} hold and also we can easily check that for all $\psi \in C^1(\Tt)$,
\[\int_{\Tt}-(mu_x)\psi_x \mbox{ }dx=0,\]
because $u_x=0$ in $\Tt$, which implies \eqref{EPFP}.
By Proposition \ref{par-uni}, $(m,\lambda)$ is uniquely determined. Moreover, all weak solutions $u\in C^{0,\frac{3}{4}}(\Tt)$ satisfy  
\[u_x=0 \quad \mbox{a.e. in } \{m>0\}=[0,1/8]\cup[3/8,5/8]\cup[7/8,1].\]
From the first equation of \eqref{example4}, all weak solutions $u$ satisfies
\[|u_x|\leq \sqrt{2\max\{-W(x),0\}} \quad \mbox{a.e. in } \Tt.\]
We consider another type weak solution satisfying 
\begin{equation}\label{ex:supersol}
\frac{1}{2}|u_x|^2+W(x)\geq0 \quad \mbox{ in } \Tt \:\:\:\mbox{in the sense of viscosity solutions},
\end{equation}which corresponds to \eqref{EPsupersol}.
In particular, due to \eqref{ex:supersol}, $u_x$ can not jump from negative value to positive value in $\{m=0\}$.
For any $\theta\in[\frac 1 8, \frac 3 8]$, we set
\begin{align}\label{theta}
u^\theta_x(x)=&\sqrt{2\max\{-W(x),0\}}\cdot \chi_{\{\frac 1 8<x<\theta\} \cup \{\frac 5 8<x<1-\theta\}}\notag \\&-\sqrt{2\max\{-W(x),0\}}\cdot \chi_{\{\theta<x<\frac 3 8\} \cup \{ 1-\theta<x<\frac 7 8\}},
\end{align}
where $\chi_A$ is the characteristic function for $A\subset \Tt$, that is, $\chi_A(x)=1$ if $a\in A$ and $\chi_A(x)=0$ if $x \not \in A$.
Set
\begin{equation}\label{func:u-theta}
u^\theta(x):= \int_0^x u^\theta_x(y) \mbox{ }dy+C,
\end{equation}
 with any $C\in \Rr$. We can easily check $u^\theta \in W^{1,4}(\Tt)\subset C^{0,\frac{3}{4}}(\Tt)$ and $(u^\theta,m,0)$ is a weak solution to \eqref{example4} and satisfies \eqref{ex:supersol}. Also, we emphasize that all weak solutions satisfying \eqref{ex:supersol} are characterized by $u^\theta$ for $\theta\in[\frac 1 8,\frac 3 8]$. 
\begin{figure}[htb!]
	\begin{centering}
		\includegraphics[width=1.0\textwidth, scale=0.4]{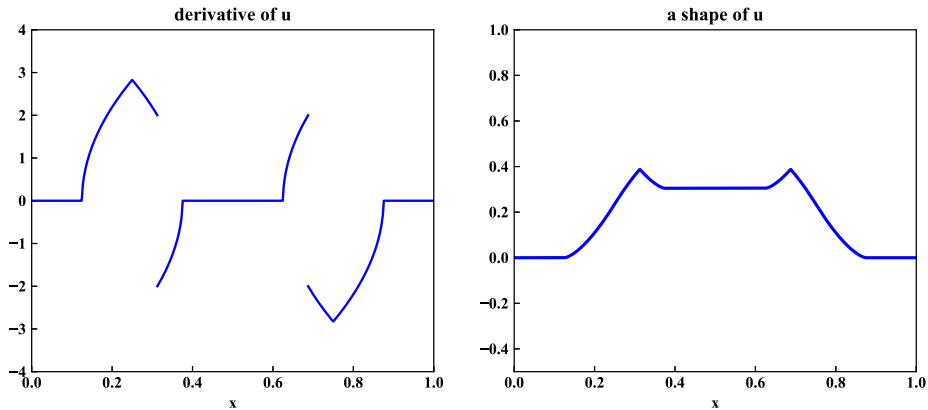}
		\caption{Plot of $u^\theta$  in \eqref{func:u-theta} with $\theta=\frac{5}{16}$ and $C=0$.}
	\end{centering}
\end{figure} 
 
\end{example}

\begin{remark}
In Example \ref{ex:1}, we point out that there exist weak solutions which do not satisfy \eqref{ex:supersol}. Indeed, $(0,m,0)$ is a weak solution, but it does not satisfy \eqref{ex:supersol}. Moreover, setting
\begin{align*}
v^\theta_x(x)=&-\sqrt{2\max\{-W(x),0\}}\cdot \chi_{\{\frac 1 8<x<\theta\} \cup \{\frac 5 8<x<1-\theta\}}\notag \\&+\sqrt{2\max\{-W(x),0\}}\cdot \chi_{\{\theta<x<\frac 3 8\} \cup \{1-\theta<x<\frac 7 8\}},
\end{align*}
we can easily check that $v^\theta$ is a weak solution for all $\theta \in [\frac 1 8, \frac 3 8]$, but it does not satisfy \eqref{ex:supersol}.
\end{remark}

\section{Selection criterion}\label{selection}
\subsection{Proof of Theorem \ref{thm:necessary}}

In this section, we prove Theorem \ref{thm:necessary}. This criterion is inspired by the work \cite{G,MT5,GMT}. 
\begin{proof}
By \eqref{pointwise}, \eqref{EPsubsol}, for almost every $x\in\{m^\epsi>0\}$, we have 
\[
f(x,m)-f(x,m^\epsi)\ge -\epsi u^\epsi+H(x,Du)-H(x,Du^\epsi) 
\ge D_pH(x,Du^\epsi)\cdot D(u-u^\epsi). 
\]
Multiplying this by $m^\epsi$ and integrating on $\Tt^d$, we get 
\[
-\epsi\int_{\Tt^d}u^\epsi m^\epsi\,dx-\int_{\Tt^d}\div(D_pH(x,Du^\epsi)m^\epsi)(u-u^\epsi)\,dx
\le 
\int_{\Tt^d}(f(x,m)-f(x,m^\epsi)) m^{\epsi}\,dx, 
\]
which implies 
\begin{equation}\label{ineq:u-uep}
\epsi\int_{\Tt^d}(u-u^\epsi)\,dx-\epsi\int_{\Tt^d}um^\epsi\,dx
\le 
\int_{\Tt^d}(f(x,m)-f(x,m^\epsi)) m^{\epsi}\,dx.  
\end{equation}

Similarly, by \eqref{subsol}, \eqref{EPpointwise}, for almost every $x\in\{m>0\}$, we have 
\[
\epsi u^\epsi+D_pH(x,Du^\epsi)\cdot D(u^\epsi-u)\le 
f(x,m^\epsi)-f(x,m). 
\]
Multiplying this by $m$ and integrating over $\Tt^d$  yields 
\begin{equation}\label{ineq:u-uep-2}
\epsi\int_{\Tt^d}u^\epsi m\,dx
\le 
\int_{\Tt^d}(f(x,m^\epsi)-f(x,m)) m\,dx.  
\end{equation}

Combining \eqref{ineq:u-uep} and \eqref{ineq:u-uep-2} together, we get 
\[
\epsi\int_{\Tt^d}(u-u^\epsi)\,dx-\epsi\int_{\Tt^d}um^\epsi\,dx
+\epsi\int_{\Tt^d}u^\epsi m\,dx
\le \int_{\Tt^d}(f(x,m^\epsi)-f(x,m)) (m-m^\epsi)\,dx\le0.   
\]
Therefore, dividing this by $\epsi>0$, and sending $\epsi=\epsi_n\to 0$, we obtain 
\[
\int_{\Tt^d}\overline{u}m\,dx-\int_{\Tt^d}\overline{u}\,dx\le 
\int_{\Tt^d}um\,dx-\int_{\Tt^d}u\,dx,  
\]
which finishes the proof. 
\end{proof}

\subsection{Convergence result}
In this section, as an application of our criterion, Theorem \ref{thm:necessary}, we show a nontrivial example that 
we get the whole convergence of the weak solution to \eqref{DP}. In Example \ref{ex:1}, we show the multiplicity of weak solutions satisfying \eqref{ex:supersol}. Here, we prove that the  minimizer of \eqref{criterion} is unique. 
In this example, it holds $pr>d$. Hence, any limits of $\langle u^{\epsi_n} \rangle$ satisfies \eqref{ex:supersol} as in Theorem \ref{result2}.
First we notice that 
\begin{equation}\label{min-theta}
\inf_{u\in \mathcal{E}}\int_{\Tt} \langle u \rangle m \mbox{ }dx= \inf_{\theta \in [\frac 1 8,\frac 3 8]}\int_{\Tt} \langle u^\theta \rangle m \mbox{ }dx,
\end{equation}
where  $u^\theta$ is the function defined by \eqref{func:u-theta}.
Because $\langle u^\theta \rangle$ is invariant with respect to adding constants, we can assume that $u^\theta(0)=0$ without loss of generality. Moreover, because $u^\theta$ and $m$ are symmetric with respect to $x=\frac 1 2$, it suffices to consider the interval $[0,\frac 1 2]$.
Recall that $m(x)=\max\{W(x),0\}$ and 
\[u^\theta_x=0 \quad \mbox{a.e. in } \{m>0\}=[0,1/8]\cup[3/8,5/8]\cup[7/8,1].\]
Thus, 
\begin{align*}
\int_0^{\frac 1 2} \langle u^\theta \rangle m \mbox{ }dx
&=\int_0^{\frac 1 2} u^\theta m \mbox{ }dx-\int_0^{\frac 1 2} u^\theta  \mbox{ }dx
=u^\theta \left(\frac3 8\right)\int_{\frac 3 8}^{\frac 1 2} m \mbox{ }dx- \int_0^{\frac 1 2} u^\theta  \mbox{ }dx\\
&= \frac{1}{4}u^\theta \left(\frac3 8 \right)-\int_0^{\frac 1 2} u^\theta  \mbox{ }dx.
\end{align*}
Then, we consider the following two cases.
\[\mbox{Case 1: }\frac 1 8 \leq \theta \leq \frac1 4,\quad \mbox{ Case 2: } \frac 1 4 \leq \theta \leq \frac3 8.\]

Case 1. Here, \eqref{theta} can be written by
\begin{align*}
u^\theta_x(x)=2\sqrt2 \left\{ \sqrt{8x-1}\left(\chi_{\{\frac 1 8<x<\theta\}}-\chi_{\{\theta<x<\frac 1 4\}} \right)-\sqrt{3-8x}\cdot \chi_{\{\frac 1 4<x<\frac 3 8\}} \right\},
\end{align*}
and $u^\theta$ is given by
\begin{align*}
u^\theta(x)=
\begin{cases}
0  &\quad  \mbox{for  } \:\:0\leq x\leq \frac 1 8\\ 
\frac {16}{3} (x-\frac 1 8)^{\frac 3 2} &\quad \mbox{for  } \:\:\frac 1 8\leq x\leq \theta \\
\frac {32}{3} (\theta-\frac 1 8)^{\frac 3 2} -\frac{16}{3}(x-\frac 1 8)^{\frac 3 2} &\quad \mbox{for  } \:\: \theta\leq x\leq \frac 1 4\\
\frac {32}{3} (\theta-\frac 1 8)^{\frac 3 2}  -\frac{32}{3}(\frac{1}{8})^{\frac 3 2}+\frac{16}{3}(\frac 3 8-x)^{\frac 3 2} &\quad \mbox{for  } \:\: \frac 1 4\leq x\leq \frac 3 8 \\
\frac {32}{3} (\theta-\frac 1 8)^{\frac 3 2}  -\frac{32}{3}(\frac{1}{8})^{\frac 3 2}&\quad \mbox{for  } \:\:\frac 3 8\leq x\leq \frac 1 2.
\end{cases}
\end{align*}
In particular, we have
\begin{align*}
\frac{\partial}{\partial \theta}u^\theta(x)=
\begin{cases}
0 &\quad \mbox{for  } \:\: 0\leq x \leq \theta\\
16  (\theta-\frac 1 8)^{\frac 1 2} &\quad \mbox{for  } \:\: \theta \leq x \leq \frac1 2.
\end{cases}
\end{align*}
Thus
\begin{align*}
\frac{\partial}{\partial \theta}\left(\int_0^{\frac 1 2} \langle u^\theta \rangle m \mbox{ }dx \right)
&= \frac{\partial}{\partial \theta} \left\{  \frac{1}{4}u^\theta\left(\frac3 8\right)-\int_0^{\frac 1 2} u^\theta  \mbox{ }dx \right\} \\
&= \frac{1}{4} \frac{\partial}{\partial \theta}  u^\theta \left(\frac3 8\right)-\int_0^{\frac 1 2} \frac{\partial}{\partial \theta} u^\theta  \mbox{ }dx=16(\theta-\frac 1 8)^{\frac 1 2}(\theta-\frac 1 4)<0,
\end{align*}
if $\theta\in (\frac 1 8,\frac 1 4)$, which implies $\int_{\Tt} \langle u^\theta \rangle m \mbox{ }dx$ is strictly decreasing in $\frac 1 8 < \theta  < \frac 1 4$.

Case 2. Here, \eqref{theta} can be written by
\begin{align*}
u^\theta_x(x)=2\sqrt2 \left\{ \sqrt{8x-1}\chi_{\{\frac 1 8<x<\frac 1 4\}}+\sqrt{3-8x}\left( \chi_{\{\frac 1 4<x<\theta \}} -\chi_{\{\theta<x<\frac 3 8\}} \right)\right\},
\end{align*}
and $u^\theta$ is given by
\begin{align*}
u^\theta(x)=
\begin{cases}
0  &\quad \mbox{for  } \:\:0\leq x\leq \frac 1 8\\ 
\frac {16}{3} (x-\frac 1 8)^{\frac 3 2} &\quad \mbox{for  } \:\: \frac 1 8\leq x\leq \frac 1 4 \\
-\frac{16}{3}(\frac 3 8-x)^{\frac 3 2}+\frac {32}{3} (\frac 1 8)^{\frac 3 2}  &\quad \mbox{for  } \:\: \frac 1 4\leq x\leq \theta\\
-\frac {32}{3} (\frac 3 8-\theta)^{\frac 3 2}  +\frac{32}{3}(\frac{1}{8})^{\frac 3 2}+\frac{16}{3}(\frac 3 8-x)^{\frac 3 2} &\quad \mbox{for  } \:\:\theta \leq x\leq \frac 3 8 \\
-\frac {32}{3} (\frac 3 8-\theta)^{\frac 3 2}  +\frac{32}{3}(\frac{1}{8})^{\frac 3 2}&\quad \mbox{for  } \:\:\frac 3 8\leq x\leq \frac 1 2.
\end{cases}
\end{align*}
In particular, we get
\begin{align*}
\frac{\partial}{\partial \theta}u^\theta(x)=
\begin{cases}
0 &\quad \mbox{for  } \:\: 0\leq x \leq \theta\\
16  (\frac 3 8-\theta)^{\frac 1 2} &\quad \mbox{for  } \:\: \theta \leq x \leq \frac1 2.
\end{cases}
\end{align*}
Thus 
\begin{align*}
\frac{\partial}{\partial \theta}\left(\int_0^{\frac 1 2} \langle u^\theta \rangle m \mbox{ }dx \right)
&= \frac{\partial}{\partial \theta} \left\{  \frac{1}{4}u^\theta\left(\frac3 8\right)-\int_0^{\frac 1 2} u^\theta  \mbox{ }dx \right\} \\
&= \frac{1}{4} \frac{\partial}{\partial \theta}  u^\theta \left(\frac3 8\right)-\int_0^{\frac 1 2} \frac{\partial}{\partial \theta} u^\theta  \mbox{ }dx=16(\frac 3 8-\theta)^{\frac 1 2}(\theta-\frac 1 4)>0,
\end{align*}
which implies $\int_{\Tt} \langle u_\theta \rangle m \mbox{ }dx$ is strictly increasing in $\frac 1 4 < \theta  < \frac 3 8$.

As observed above, the minimization \eqref{min-theta} attains at $\theta=\frac{1}{4}$ uniquely.
Note that the limit $\tilde u$ of $\langle u^\epsi \rangle$ satisfies $\int_{\Tt^d} \tilde u \mbox{ }dx=0$ because $\int_{\Tt^d} \langle u^\epsi \rangle \mbox{ }dx=0$.
In conclusion, we get the following result.

\begin{proposition}\label{ex:1-2}
Let $W$ be the function defined in Example \ref{ex:1}, and   
let  $(u^\epsi,m^\epsi)$ be the weak solution to
\begin{equation*}\label{example3}
        \begin{cases}
                \:\: \epsi u^\epsi+\frac{1}{2}|u_x^\epsi|^2+W(x)=m^\epsi &\quad\text{in}\  \Tt, \\
                \:\: \epsi m^\epsi-(m^\epsi u^\epsi_x)_x= \epsi &\quad \text{in} \ \Tt.
        \end{cases}
\end{equation*}
Let $\tilde u$ be defined by \eqref{func:u-theta} with $\theta=\frac{1}{4}$ and choosing $C\in \Rr$ such that $\int_{\Tt^d} \tilde u \mbox{ }dx=0$. 
Then, $\langle u^\epsi \rangle \to \tilde u$ as $\epsi \to 0$ uniformly on $\Tt$. 
 \end{proposition}
\begin{figure}[htb!]
	\begin{centering}
		\includegraphics[width=1.0\textwidth, scale=0.4]{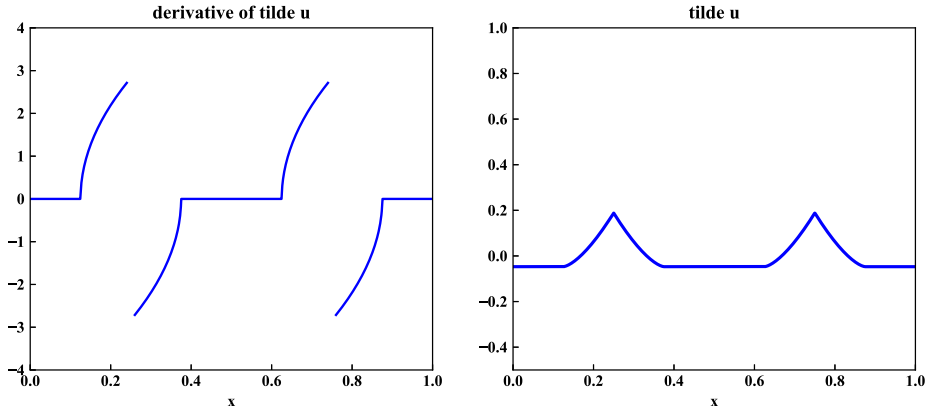}
		\caption{The limit of $\langle u^\epsi \rangle$  in Proposition \ref{ex:1-2}.}
	\end{centering}
\end{figure} 

\section{uniqueness set}\label{uniquenessstructure}

In this section, we consider a uniqueness set for weak solutions to the ergodic problem \eqref{EP}. We call  $\mathcal{Z}\subset \Tt^d$ a {\it uniqueness set} if  two weak solutions $(u_1,m,\lambda),(u_2,m,\lambda) \in W^{1,pr}(\Tt^d) \times L^q(\Tt^d) \times \Rr$ with \eqref{EPsupersol} satisfy $u_1=u_2$ on $\mathcal{Z}$, then $u_1=u_2$ on  $\Tt^d$.

\subsection{Classical comparison principle.}
We prove that $\mathcal{Z}$ given by \eqref{uniquenessset} is a uniqueness set of the ergodic problem \eqref{EP}.
We first notice that $m$ may not be continuous in general, so the notion of viscosity solutions may not work in $\Tt^d$. However, if $\mathcal{Z}^c:= \Tt^d \setminus \mathcal{Z} \neq \emptyset$, then $u$ satisfies
\[H(x,Du) \leq f(x,0)+\lambda \quad \mbox{ a.e. in } \mathcal{Z}^c. \]
Then, $u\in W^{1,\infty}(\mathcal{Z}^c)$.
Due to the convexity of Hamiltonian $u$ satisfies
\[H(x,Du) \leq f(x,0)+\lambda \quad \mbox{ in } \mathcal{Z}^c, \quad \mbox{ in the viscosity sense}. \]
Therefore, we can use a standard technique in the theory of viscosity solutions. Indeed, by a simple adaptation of the argument in \cite{Ishii}, we can easily prove the following comparison result. We give it for the completeness for the paper.

\begin{proof}[Proof of Theorem {\rm\ref{unique1}}]
Let $0<\mu<1$. Suppose that 
\[ \max_{x\in \Tt^d} \left\{ \mu u(x)-v(x)\right\}>0,\]
and this maximum is attained at some $x_0 \in \mathcal{Z}^c$. For $\delta>0$ we define $\Psi: \Tt^d \times \Tt^d \to \Rr$ by
\begin{equation*}
\Psi(x,y):=\mu u(x)-v(y)-\frac{|x-y|^2}{2\delta}-|x-x_0|^2.
\end{equation*}
Take $(x_\delta,y_\delta)$ so that $\Psi(x_\delta,y_\delta)=\max_{\Tt^d \times \Tt^d} \Psi$.
By a standard argument, we can prove
$x_\delta \to x_0$ and $y_\delta \to x_0$ as $\delta \to 0$.

Let $U_{x_0}$ be an open neighborhood of $x_0$ with $U_{x_0} \subset \mathcal{Z}^c$. Then, by \eqref{EPsubsol}, $u$ is a viscosity subsolution to
\[ H(x,Du)\leq f(x,0)+\lambda  \quad \mathrm{in}\ U_{x_0} .\]
Taking $\delta>0$ so small that $x_\delta, y_\delta \in U_{x_0}$, by the definition of viscosity solutions, we get
\begin{equation}\label{musub}
\mu H\left(x_\delta, \frac{p_\delta+2(x_\delta-x_0)}{\mu} \right)\leq \mu f(x_\delta,0)+\mu \lambda
\end{equation}
and
\begin{equation}\label{musuper}
H(y_\delta,p_\delta)\geq f(y_\delta,0)+\lambda,
\end{equation}
where $p_\delta:=\frac{x_\delta-y_\delta}{\delta}$. 
By the coercivity of $H(x,\cdot)$, inequality \eqref{musub} implies that $|p_\delta|\leq R$ for some $R>0$ independent of $\delta$. Subtracting \eqref{musub} from \eqref{musuper} yields
\begin{equation*}\label{gggg}
\mu H\left(x_\delta,\frac{p_\delta+2(x_\delta-x_0)}{\mu} \right) -H(y_\delta,p_\delta) \leq \mu f(x_\delta,0)-f(y_\delta,0)+(\mu-1)\lambda.
\end{equation*}
By passing to a subsequence if necessary, we have 
\[p_{\delta_j} \to p \in B(0,R).\]
Sending  $\delta=\delta_j \to 0$ yields
\[ (\mu-1)\{ f(x_0,0)+\lambda\} \geq \mu H\left(x_0,\frac{p}{\mu}\right)- H(x_0,p)\geq -(1-\mu)H(x_0,0),\]
by the convexity of Hamiltonian, which implies 
\[H(x_0,0)-\lambda-f(x_0,0) \geq0.\]
This contradicts a fact $x_0 \in \mathcal{Z}^c$.
\end{proof}

\begin{example}
 Let $H(x,p)=\frac{1}{2}|p|^2+V(x)$ for $V\in C(\Tt^d)$ and $f\in C([0,\infty))$. Then, 
\begin{equation*}\label{example2}
        \begin{cases}
                \:\: \frac{1}{2}|Du|^2+V(x)=f(m)+\lambda &\quad\text{in}\  \Tt^d, \\
                \:\: -\mathrm{div}(mDu)= 0 &\quad \text{in} \ \Tt^d.
        \end{cases}
\end{equation*}
Then, for all $c\in \Rr$, $(c,m,\lambda)$ is a weak solution, where $m$ is uniquely determined by
\begin{align*}
m(x):=\max\{f^{-1}(-\lambda+V(x)),0\} 
\end{align*}
and we choose $\lambda \in \Rr$ such that $\int_{\Tt^d} m(x)=1$. Thus, \eqref{uniquenessset} becomes
\begin{equation*}
\mathcal{Z}=\{ x\in \Tt^d \:|\: -\lambda+V(x)-f(0)\geq0\}.
\end{equation*}
\end{example}

Combining Proposition \ref{par-uni} and Theorem \ref{unique1}, we get a sufficient condition to obtain the whole convergence of $\langle u^\epsi \rangle$.
\begin{corollary}
Let $c\in \Rr$ and assume that $(c,m,\lambda)$ is a weak solution to ergodic problem \eqref{EP}. Furthermore assume that the set defined by
\[ \{x\in \Tt^d \:|\: -\lambda+H(x,0)-f(x,0)>0\}\]
is connected and satisfies
\[ \{x\in \Tt^d \:|\: -\lambda+H(x,0)-f(x,0)\geq0\}=\overline{\{x\in \Tt^d \:|\: -\lambda+H(x,0)-f(x,0)>0\} } .\]
 Then, weak solution $u\in W^{1,\infty}(\Tt^d)$ to \eqref{EP} with \eqref{EPsupersol} is unique up to constants. 
 
Moreover,  if $pr>d$, then $\langle u^\epsi \rangle$ uniformly converges to a unique limit on $\Tt^d$ as $\epsi \to 0$.
\end{corollary}

\begin{proof}
By Proposition \ref{m:const}, we have $m\in C(\Tt^d)$, which implies $u\in W^{1,\infty}(\Tt^d)$. 
Because $(c,m,\lambda)$ is a weak solution, it holds that
\begin{align*}
\{x\in \Tt^d \:|\: m(x)>0\}
&=\{x\in \Tt^d \:|\: f^{-1}(x,-\lambda+H(x,0))>0\}\\
&=\{x\in \Tt^d \:|\: -\lambda+H(x,0)-f(x,0)>0\}.
\end{align*}
Let $(u_1,m,\lambda)$ and $(u_2,m,\lambda)$ be weak solutions to \eqref{EP} and satisfy \eqref{EPsupersol}.  By Proposition \ref{par-uni}, $Du_1=Du_2=0$ in $\{m>0\}$. Because $\{m>0\}$ is connected, there exists a unique constant $M\in \Rr$ such that $u_1=u_2+M$ in $\{m>0\}$. Since
\begin{align*}
\overline{\{x\in \Tt^d \:|\: m(x)>0\}}
&=\overline{\{x\in \Tt^d \:|\: -\lambda+H(x,0)-f(x,0)>0\}}\\
&=\{x\in \Tt^d \:|\: -\lambda+H(x,0)-f(x,0)\geq0\},
\end{align*}
it holds that  $\mathcal{Z}=\overline{ \{m(x)>0 \}}$. Hence, by Theorem \ref{unique1}, we have $u_1=u_2+M$ in $\Tt^d$.

Under assumption $pr>d$, $\langle u^\epsi \rangle$ uniformly converges to $u$, which is a unique (up to constants) weak solution to \eqref{EP} and satisfies \eqref{EPsupersol}, along subsequences. Because $\int_{\Tt^d} \langle u^\epsi \rangle\mbox{ }dx=0$, one has $\langle u^\epsi \rangle \to \bar u$ uniformly as the whole sequence, where $\bar u$ satisfies $\int_{\Tt^d} \bar u \mbox{ }dx=0$. 
\end{proof}

\subsection{Comparison in terms of Mather measures}
In this section, we first prove Theorem \ref{weakvis}, and as a corollary, we give a comparison principle in terms of Mather measures. 


%

\begin{proof}[Proof of Theorem {\rm\ref{weakvis}}]
We first prove (i). 
We notice that under the continuity assumption on $m$ we can easily see that 
$\lambda$ coincides with the ergodic constant for the Hamilton-Jacobi equation.  
Thus, \eqref{EPHJ} admits at least one viscosity solution $v\in \mathrm{Lip}(\Tt^d)$. 
Note that $v$ satisfies
\begin{equation*}
H(x,Dv)=f(x,m)+\lambda \quad\text{a.e.  in}\  \Tt^d.
\end{equation*}
Thus, we obtain
\begin{align*}
\mathcal{A}(v,\lambda)
&=\int_{\Tt^d} F^*(x,-\lambda+H(x,Dv))+\lambda \mbox{ }dx= \int_{\Tt^d} F^*(x,f(x,m))+\lambda \mbox{ }dx\\
&=\int_{\{m>0\}} F^*(x,-\lambda+H(x,Du))+\lambda \mbox{ }dx+\int_{\{m=0\}} F^*(x,f(x,0))+\lambda \mbox{ }dx.
\end{align*} 
As in \eqref{Fstar}, because 
\[F^*(x,-\lambda+H(x,Du))=F^*(x,f(x,0))=0  \quad\text{on}\  \{m=0\}, \]
we have
\[\mathcal{A}(v,\lambda)=\int_{\Tt^d} F^*(x,-\lambda+H(x,Du))+\lambda \mbox{ }dx=\mathcal{A}(u,\lambda).\]
Therefore, $(v,\lambda)$ is a minimizer of $\inf_{(\phi,c)\in W^{1pr}\times \Rr}\mathcal{A}(\phi,c)$. As in Proposition \ref{DPchar}, $(v,m,\lambda)$ is a weak solution to \eqref{EP}. Because $f(x,\cdot)$ is increasing, $v$ satisfies \eqref{EPsupersol} in the sense of viscosity solutions.

Next, we prove (ii). In light of \eqref{EPsubsol} and $m\in C(\Tt^d)$, we have $u\in W^{1,\infty}(\Tt^d)$. 
 Due to the convexity of $H(x,\cdot)$, we can easily check $u$ is a viscosity subsolution to \eqref{EP}. We prove that $u$ is a viscosity supersolution by contradiction. Suppose that there exists $x_0 \in \Tt^d$ and $\phi \in C^1(\Tt^d)$ such that 
\[(u-\phi)(x_0)< (u-\phi) (x) \quad \mbox{ for all } x\in U_{x_0},\]
and
\begin{equation}\label{ccd}
H(x_0,D\phi(x_0))<f(x_0,m(x_0))+\lambda,
\end{equation}
where $U_{x_0}$ is a  sufficiently  small open neighborhood of $x_0$. Because $u$ satisfies \eqref{EPsupersol}, it is enough to consider the case $x_0 \in \{m>0\}$ and $U_{x_0} \subset \{m>0\}$.

Here, let $v\in \mathrm{Lip}(\Tt^d)$ be a viscosity solution to \eqref{EPHJ}. Then, as we proved the above, $(v,m,\lambda)$ is a weak solution to \eqref{EP}. Note that $Du=Dv$ a.e. in $\{m>0\}$ due to Proposition \ref{par-uni}. Thus, there exists a constant $\hat c= u(x)-v(x)$ for all $x\in U_{x_0}$. In particular, we have
\[(v+\hat c-\phi)(x_0)< (v+\hat c-\phi) (x) \quad \mbox{ for all } x\in U_{x_0}.\] 
 Because $v+\hat c$ is a viscosity solution to \eqref{EPHJ},  it holds
 \[H(x_0,D\phi(x_0))\geq f(x_0,m(x_0))+\lambda,\]
which contradicts \eqref{ccd}.
\end{proof}

\begin{corollary}\label{unique2}
Let $(u_1,m,\lambda),(u_2,m,\lambda)\in W^{1,pr}(\Tt^d)\times L^q(\Tt^d) \times \Rr$ be weak solutions to ergodic problem \eqref{EP} and satisfy \eqref{EPsupersol}. Assume that $m\in C(\Tt^d)$. If
\begin{equation*}
\iint_{\Tt^d \times \Rr^d} u_1 \mbox{ }d\mu(x,v) \leq  \iint_{\Tt^d \times \Rr^d} u_2 \mbox{ }d\mu(x,v) \quad \mbox{ for all } \mu \in \mathcal{M},
\end{equation*}
 then, $u_1 \leq u_2$ in $\Tt^d$. 
 Here, we denote by  $\mathcal{M} \subset \mathcal{P}(\Tt^d \times \Rr^d)$ all minimizing measures of
\begin{equation*}\label{actionmin}
\inf_{\mu \in \mathcal{H}} \iint_{\Tt^d \times \Rr^d} L(x,v)\mbox{ }d\mu(x,v),
\end{equation*}
where we set $L(x,v):=\sup_{p\in \Rr^d}\{v\cdot p-H(x,p)+f(x,m)\}$, and 
\begin{equation*}
\mathcal{H}:=\left\{ \mu \in \mathcal{P}(\Tt^d \times \Rr^d) \mid \iint_{\Tt^d \times \Rr^d} v\cdot D\phi \mbox{ }d\mu(x,v)=0 \quad \text{for all} \ \phi \in C^1(\Tt^d) \right\}.
\end{equation*}
\end{corollary}

\begin{proof}
By Theorem \ref{weakvis}, $u_1$ and $u_2$ are viscosity solutions to \eqref{EPHJ}. Thus, this is a straightforward result of \cite[Theorem 1.1]{MT6}. 
\end{proof}

\appendix

\section{Sobolev Estimate}

In this appendix, we obtain $H^1$-estimate for $m^\epsi$ inspired by the works in \cite{GM,PS,Santambrogio}. In this section, we additionally assume that $H^*$ and $F$ are twice-differentiable with respect to the first variable.  Here is the goal of this appendix.

\begin{proposition}\label{H1esti}
Assume that there exist $J, \tilde{J}: [0,\infty) \to \Rr$ and $c_0>0$, such that for all $m,a \in [0,\infty)$,
\begin{equation}\label{hypJ}
F(x,m)+F^*(x,a)-ma\geq c_0|J(m)-\tilde{J}(a)|^2. 
\end{equation}
Further, assume that there exists $C>0$ such that 
\begin{equation}\label{H^*xx}
 \sum_{i,j=1}^d  \frac{\partial^2 H^*}{\partial x_i \partial x_j}(x,p)h_i h_j \leq C\left( |p|^{r'}+1\right)|h|^2, 
\end{equation}
\begin{equation}\label{Fxx}
 \sum_{i,j=1}^d  \frac{\partial^2F}{\partial x_i \partial x_j}(x,m) h_i h_j \leq C\left( |m|^{q}+1\right)|h|^2, 
\end{equation}
for all $h=^t\!(h_1,\ldots,h_d) \in \Rr^d$, $(x,p)\in \Tt^d \times \Rr^d$ and $ m \geq0$. 
Let $(m^\epsi,w^\epsi)\in K_{\epsi}$ be the minimizer of \eqref{Bmin}. 
There exists a constant $C>0$, which is independent of $\epsi$, such that 
\[ \|J(m^\epsi)\|_{H^1(\Tt^d)}\leq C.\]
\end{proposition}
%

\begin{lemma}\label{Bhesti}
Let us define $B^\epsi: \Tt^d \to \Rr$ by 
\begin{align*}
B^\epsi(h):=\int_{\Tt^d} m^\epsi(x)H^*\left(x+h,\frac{-w^\epsi(x)}{m^\epsi(x)}\right)+F(x+h,m^\epsi(x))\,dx, 
\end{align*}
where $(m^\epsi,w^\epsi)$ is obtained in Proposition \ref{minexist}.
There exists a constant $C>0$ such that for all $h\in \Tt^d$ and $\epsi>0$, 
\begin{equation*}
B^\epsi(h) \leq B^\epsi(0)+C|h|^2 .
\end{equation*}
\end{lemma}

\begin{proof}
In the proof, we write $f_h(x):= f(x-h)$ for all $h\in \Tt^d$ and $f:\Tt^d \to \Rr$. 
We first notice that $B^\epsi: \Tt^d \to \Rr$ takes a minimum at $h=0$. 
Indeed, since $(m^\epsi_h,w^\epsi_h) \in K_\epsi$, we have 
\begin{equation}\label{DBh} 
B^\epsi(h)=\mathcal{B} (m^\epsi_h,w^\epsi_h) \geq \mathcal{B} (m^\epsi,w^\epsi) =B^\epsi(0).
\end{equation}

Let $y \in \Tt^d$ and $h=(h_1,\ldots,h_d)\in \Tt^d$. By Taylor's expansion, there exists a constant $\theta \in (0,1)$ such that
\begin{align*}
&B^\epsi(y+h)-B^\epsi(y)\\
&=\int_{\Tt^d} m^\epsi  D_x H^* (x+y, -\frac{w^\epsi}{m^\epsi})\cdot h+D_xF(x+y,m^\epsi)\cdot h  \mbox{ }dx\\
&\quad \:\: +\frac{1}{2}\int_{\Tt^d} \sum_{i,j=1}^d \left\{ m^\epsi \frac{\partial^2 H^*}{\partial x_i \partial x_j}(x+y+\theta h, -\frac{w^\epsi}{m^\epsi})+\frac{\partial^2 F}{\partial x_i \partial x_j} (x+y+\theta h,m^\epsi) \right\} h_i h_j \mbox{ }dx.
\end{align*}
It follows that $B^\epsi \in C^1(\Tt^d)$ with 
\[ DB^\epsi(y)=\int_{\Tt^d} m^\epsi D_xH^*(x+y,-\frac{w^\epsi}{m^\epsi})+D_xF(x+y,m^\epsi)\mbox{ }dx.\]
In view of \eqref{DBh}, $DB^\epsi(0)=0$. Hence, it follows from  \eqref{H^*xx} and \eqref{Fxx} that 
\begin{align}\label{Taylor}
&B^\epsi(h)-B^\epsi(0) \notag \\
&= \frac{1}{2}\int_{\Tt^d} \sum_{i,j=1}^d \left\{ m^\epsi \frac{\partial^2 H^*}{\partial x_i \partial x_j}(x+y+\theta h, -\frac{w^\epsi}{m^\epsi})+\frac{\partial^2 F}{\partial x_i \partial x_j} (x+y+\theta h,m^\epsi) \right\} h_i h_j \mbox{ }dx \notag \\
&\leq \frac{C|h|^2}{2} \int_{\Tt^d} m^\epsi \left|\frac{w^\epsi}{m^\epsi} \right|^{r'}\mbox{ }dx+  |m^\epsi|^q+1 \mbox{ }dx. 
\end{align}
Noting that $\|m^\epsi\|_{L^q(\Tt^d)}$ is bounded uniformly in $\epsi$ by Proposition \ref{compact}, and as in estimate \eqref{mqesti}, we have 
\begin{equation*}
 \int_{\Tt^d}m^\epsi \left|\frac{w^\epsi}{m^\epsi}\right|^{r'} \mbox{ }dx \leq C\left( \mathcal{B}(1,0)+1\right).
\end{equation*}
Hence,  combining the above with \eqref{Taylor}, we have  \eqref{Bhesti}.
\end{proof}

\begin{lemma}\label{Jesit} 
We have 
\begin{equation*}
\mathcal{A}^\epsi(u)+\mathcal{B}(m,w)\geq c_0 \int_{\Tt^d} |J(m)-\tilde{J}(\epsi u+H(x,Du))|^2 dx
\end{equation*}
for all $u\in W^{1,pr}(\Tt^d)$, and $(m,w) \in K_\epsi$.

\end{lemma}

\begin{proof}
Using \eqref{hypJ}, and the convexity of $H^\ast(x,\cdot)$, we have
\begin{align*}
&\mathcal{A}^\epsi(u)+\mathcal{B}(m,w)\\
\geq& 
c_0 \int_{\Tt^d} |J(m)-\tilde{J}(\epsi u+H(x,Du))|^2 dx \\
&+ \int_{\Tt^d} m\left(H(x,Du)+H^*\left(x, -\frac{w}{m}\right)\right)
-\epsi u+\epsi u m \, dx\\
\geq& 
c_0 \int_{\Tt^d} |J(m)-\tilde{J}(\epsi u+H(x,Du))|^2 \,dx 
+ \int_{\Tt^d} -Du\cdot w-\epsi u+\epsi u m \,dx\\
=& 
c_0 \int_{\Tt^d} |J(m)-\tilde{J}(\epsi u+H(x,Du))|^2 \,dx. 
\qedhere   
\end{align*}
\end{proof}

\begin{proof}[Proof of Proposition {\rm\ref{H1esti}}]
Let $u^\epsi$ be a minimizer of \eqref{Aepsimin}. 
Using Lemmas \ref{Jesit}, \ref{Bhesti} and Proposition \ref{prop:dual}, we get 
\begin{align*}
&\frac{1}{2}\|J(m^\epsi)-J(m^\epsi_h)\|^2_{L^2(\Tt^d)}\\
&\leq \|J(m^\epsi)-\tilde{J}(\epsi u^\epsi +H(x,Du^\epsi))\|^2_{L^2} +\|\tilde{J}(\epsi u^\epsi +H(x,Du^\epsi))-J(m^\epsi_h)\|^2_{L^2}\\
&\leq 
\frac{1}{c_0}\{\mathcal{A}^\epsi(u^\epsi)+\mathcal{B}(m^\epsi,w^\epsi)\}+\frac{1}{c_0}\{\mathcal{A}^\epsi(u^\epsi)+\mathcal{B}(m^\epsi_h,w^\epsi_h)\}
\\
&=\frac{1}{c_0}\{-\mathcal{B}(m^\epsi,w^\epsi)+\mathcal{B}(m^\epsi_h,w^\epsi_h)\}
 = \frac{1}{c_0}\{B(h)-B(0)\}\leq C|h|^2,
\end{align*}
which gives the uniform boundness of $J(m^\epsi)$ in $H^1(\Tt^d)$.
\end{proof}

\bigskip
\noindent
\textbf{Acknowledgement. }
The authors would like to thank Hung V. Tran and Alp\'{a}r R. M\'{e}sz\'{a}ros for helpful comments and suggestions. 


\def\cprime{$'$}

\end{document}